\PassOptionsToPackage{x11names}{xcolor}
\documentclass[12pt,a4paper,twoside]{article}

\usepackage{array}
\usepackage{amsmath}
\usepackage{tabularx}
\usepackage{txfonts}
\usepackage{enumerate}
\usepackage{cite}
\usepackage{float}
\usepackage{xspace}
\usepackage[x11names]{xcolor}
\usepackage{multirow}

\usepackage{tikz}
\usepgflibrary{shapes}
\usetikzlibrary{spy}
\usepackage{pgfplots}

\usepackage[T1]{fontenc}
\usepackage[cp1250]{inputenc}

\usepackage[margin=2cm]{geometry}

\usepackage{changes}
\usepackage[colorinlistoftodos, textwidth=30mm, shadow]{todonotes}

\colorlet{cmarek}{Chartreuse3}
\colorlet{cmonika}{DarkOrchid1}

\DeclareMathOperator{\e}{\mathrm{e}}
\DeclareMathOperator{\dd}{\mathrm{d\!}}

\newcommand{\etal}{\textit{et al.}\xspace}

\renewcommand{\Re}{\textrm{Re}}
\renewcommand{\Im}{\textrm{Im}}

\newcommand{\N}{\varmathbb{N}}
\newcommand{\R}{\varmathbb{R}}
\newcommand{\C}{\varmathbb{C}}
\newcommand{\Cb}{\mathcal{C}}
\newcommand{\Kb}{\mathcal{K}}

\newcommand*{\kryt}{\textrm{cr}}

\newtheorem{thm}{Theorem}[section]

\newtheorem{prop}[thm]{Proposition}

\newtheorem{rem}[thm]{Remark}

\newenvironment{proof}[1][]{\par\noindent\textbf{Proof #1: }}{\hfill\rule{1.3ex}{1.3ex}\par}

\numberwithin{equation}{section}

\begin{document}
\pagestyle{myheadings}
\markright{Bodnar, M. \& Piotrowska,~M.J.: Family of angiogenesis models with distributed delay}

\noindent\begin{tabular}{|p{\textwidth}}
	\Large\bf Analysis of the family of angiogenesis models with distributed time delays \\\vspace{0.01cm}
    \it Bodnar, M.$^{*,1}$ \& Piotrowska, M.J.$^{*,2}$\\\vspace{0.02cm}
\it\small $^*$Institute of Applied Mathematics and Mechanics,\\
\it\small University of Warsaw, Banacha 2, 02-097 Warsaw, Poland\\\vspace{0.01cm}
\small  $^1$\texttt{mbodnar@mimuw.edu.pl}, $^2$\texttt{monika@mimuw.edu.pl}\\
    \multicolumn{1}{|r}{\large\color{orange} Research Article} \\
	\\
	\hline
\end{tabular}
\thispagestyle{empty}

\tableofcontents
\noindent\begin{tabular}{p{\textwidth}}
	\\
	\hline
\end{tabular}

\begin{abstract}
In the presented paper a~family of angiogenesis models, that is a~generalisation of the Hahnfeldt \etal model is  proposed.
Considered family of models consists of two differential equations
with distributed time delays. 
The global existence and the uniqueness of the solutions are proved. Moreover, the stability 
of the unique positive steady state is examined in the case of Erlang 
and piecewise linear delay distributions.
Theorems guaranteeing the existence of  stability switches and 
occurrence of  Hopf bifurcations are proved. Theoretical results are illustrated by numerical analysis performed
for parameters estimated by Hahnfeldt \etal (\textit{Cancer Res.}, 1999). 
\end{abstract}

Keywords: 
delay differential equations, distributed delay, stability analysis, Hopf bifurcation, angiogenesis


\section{Introduction}

Angiogenesis is a~very complex process which accompanies us through our whole life starting from the development of the 
embryo and ending with wound healing in an advanced age  since it is a~process of blood vessels formation from pre-existing structures.  
Clearly,  it is often considered as a~vital process involved in organism growth and development. On the other hand,
it can also be a~pathological process since it may promote the growth of solid tumours. Indeed, in the first stage of 
solid tumour development cells create multicellular spheroid (MCS) --- a~small spherical aggregation. For a~fast growing tumour (comparing to the healthy tissue) tumour's cells located in the centre of MCS receive 
less and less nutrients (such as glucose and oxygen) since they are only supplied through the diffusion of that substances 
from the external vessels. Hence, when MSC reaches a~certain size (usually around 2--3 mm in diameter, 
\cite{folkman1971tumor,GIMBRONE1972}) two processes are observed: the saturation of growing cellular mass, 
and the formation of necrotic core in the centre of MCS. 
The tumour cells that are poorly nourished secrete number of angiogenic factors, e.g. FGF, VEGF, VEGFR, Ang1 and Ang2, 
which promote the proliferation and the differentiation of endothelial cells, smooth muscle cells, fibroblasts and also 
stabilise new created vessels initiating the process of angiogenesis. 
From that point of view, angiogenesis can be also considered as an essential step in the  tumours transition from less harm for hosts avascular forms to cancers that are able to metastasise and finally cause the lethal outcome of the disease. 

In~\cite{folkman1971tumor} Folkman showed that the growth of solid tumours strongly depends on the amount of blood vessels that are induced to grow by tumours. He summarised that, if the tumour could be stopped from  its own blood supply, it would wither and die, and nowadays  he is considered as a~father of an
anti-angiogenic therapy the approach that aims to prevent the tumour from developing its own blood supply system.

On the other hand, angiogenesis has also an essential role in the tumours treatment during chemotherapy  since 
anti-cancer drugs distributed with blood need to be delivered to the tumour and efficiently working vessels net 
allows the drugs to penetrate better the tumour structure. One also needs to take into account the fact that the vessel's 
net developed due to  angiogenesis initiated by tumour cells is not as efficient as the net in healthy tissue 
e.g. consists of loops. That causes large difficulties during the treatment of tumours because drugs transport is often 
ineffective.
Because of the reasons explained above, the process of angiogenesis is very important for the solid tumour 
growth and also for the anti-tumour treatment. Hence, a~
number of authors described that process using various mathematical models: macroscopic
\cite{Anderson1998,Aubert2011,bodnar13mbe_angio,Chaplain1990, Chaplain1993,donofrio2004,donofrio2006, onf2009, 
Hahnfeldt1999}, individual-based models \cite{Chaplain2012,Watson2012} or hybrid models  \cite{Chaplain2012a,McDougall2012}.
One of the most well recognised angiogenesis models was proposed by 
Hahnfeldt~\etal~\cite{Hahnfeldt1999} later 
on studied in~\cite{donofrio2004}. Among others, models with discrete delays describing the same angiogenesis 
process were recently considered in~\cite{bodnar13mbe_angio} and~\cite{mbuf07jbs}. Models investigating 
the tumour development and angiogenesis in the context of anti-angiogenic therapy, chemotherapy 
or radiotherapy were also considered 
in~\cite{mbuf09appl,Alberto2007prE,ergun03bmb,McDougall2002,Orme1997,angioMBE,Sachs2001}, in 
the context of optimal treatment schedules~\cite{ledz2007,ledz2008, ledz2009, 
swier2006,swier1,Swierniak2012} or 
in context of vessels maturation in~\cite{Arakelyan2005, arakelyan}.

Let consider the family of the angiogenesis models that consists of two differential equations 
in the following form
\begin{subequations}\label{model}
	\begin{align}
		\frac{\dd}{\dd t}p(t)&=rp(t)h\left(\frac{p(t-\tau_1)}{q(t-\tau_1)}\right),
					\label{model:1}\\
		\frac{\dd}{\dd t}q(t)&=q(t)\Bigg(b\left(\frac{p(t-\tau_2)}{q(t-\tau_2)}\right)^{\alpha} -a_Hp^{2/3}(t) 
					- \mu \Bigg) ,\label{model:2} 
	\end{align}
\end{subequations}
where  variables $p(t)$ and 
$q(t)$ represent the tumour volume at time $t$, and the 
maximal  tumour size for which tumour cells can be nourished by the present vasculature. Thus, 
the ratio $\frac{p}{q}$ is interpreted as the measure of vessels density.
The function $h$ reflects dependence of tumour growth rate on
the vascularisation. We assume $h$ is decreasing.
In the literature, in the context of  model~\eqref{model}, it is usually assumed that the 
tumour growth is governed by logistic, generalised logistic (see~\cite{donofrio2004, 
donofrio2006,onf2009}) or by  Gompertzian (see~\cite{Hahnfeldt1999,donofrio2004}) law. 
The parameter $r$ reflects the maximal tumour growth rate. The delays $\tau_1$ and $\tau_2$ present in the system  
represent time lags in the processes of the tumour growth and the vessels formation, respectively. 
Moreover, it is assumed that the population of the endothelial cells depends on both: the stimulation process 
initiated by poorly nourished tumour cells and the inhibiting factors secreted by tumour cells causing the lost of vessels. The parameter $\alpha$ reflects the strength of the dependence of the vessel dynamics on the 
ratio  $\tfrac{p}{q}$.  Parameters $a_H$ and $b$  are proportionality parameters of inhibition 
and stimulation of the angiogenesis process, respectively. 
We consider two processes, of which one appears on the surface of 
a~sphere and the other the inside the sphere, and due to the spherical symmetry assumption
we suppose that the ratio of the surface process depends proportionally to the tumour surface and 
it is  represented by the exponent~$2/3$.
Thus, the term $-a_Hq(t)p^{\frac{2}{3}}(t)$ appears 
in Eq.~\eqref{model:2}.
In model~\eqref{model} the spontaneous loss of functional vasculature is represented 
by the term $-\mu q(t)$. 
In~\cite{Hahnfeldt1999}, the parameter $\mu$ was estimated to be equal to zero. 
On the other hand, the term $-\mu q(t)$ for $\mu>0$ can be also interpreted as a~constant continuous anti-angiogenic 
treatment, hence it is also considered in the presented study. 
For detailed derivation of system~\eqref{model} without delays and with logarithmic function $h$ from 
a reaction-diffusion system see~\cite{Hahnfeldt1999}.

Model~\eqref{model} 
with  Gompertzian tumour growth, $\alpha=1$, and without time delays 
was firstly proposed by Hahnfeldt~\etal in \cite{Hahnfeldt1999}. 
The modification of the Hahnfeldt~\etal model was considered by various authors. d'Onofrio and 
Gandolfi~\cite{donofrio2004} proposed system of ordinary differential equations (ODEs) based on Hahnfeldt~\etal
idea but with $h$ being linear or logarithmic function, and with $\alpha=0$, that is when the dynamics of the second 
variable  of the model does not depend on the vascularisation of the tumour. Later in~\cite{mbuf07kkzmbm} 
Bodnar and Fory\'{s} introduced discrete time delays into the model with the Gompertzian tumour growth and 
the family of models for parameter $\alpha$ from the interval $[0,1]$ was 
considered. The analysis was later extended in~\cite{mbuf09appl}. In the same time, 
d'Onofrio and Gandolfi in~\cite{onf2009} studied the model for $\alpha=0$ with discrete time delays and 
different functions $h$. The analysis of the family of the models with discrete delays and the  Gompertzian tumour growth was extended 
in~\cite{Piotrowska2011,Forys2013} where the directions of the existing Hopf bifurcations were studied.

However, in all these papers only discrete time delays were considered. 
One of the reasons is the fact that mathematical analysis of such models is easier than the models with distributed delays.
Of course, in reality the duration of a~process is never constant and it usually fluctuates around some value. Thus, 
we believe that delay is somehow distributed around some average value. This is the main reason
why   in this paper we study the modification of the family of 
models~\eqref{model} in which the delays are distributed around theirs average values $\tau_1$ and $\tau_2$ 
instead of being concentrated in these points. On the other hand, we should also point out, 
that it might be difficult to estimate a~particular delay distribution form experimental data.

Up to our knowledge such modification of the family of models has not 
been considered yet. We are also not aware of any result considering linear systems similar to the linearised 
system of the model with distributed delays considered in this paper. 
In fact known analytical results regarding the stability of trivial steady state for  linear equations with 
distributed delays are rather limited, see e.g. \cite{Kolmanovskii1992} and references therein. Most results concern a
single equation (see~\cite{Bernard2001,Campbell2009mmnp,Huang2012dcds}, and references therein). In some papers a~second 
order equation or a~system of equations with distributed delays are considered (see \cite{Kiss2010dcds} for the second 
order linear equation, and~\cite{Faria2008jde} for the Lotka-Volterra system). However, in these two cases the
time delays are  only finite and the delayed terms do not appear on the diagonal of the stability matrix, and hence 
these results cannot be applied directly to the system
considered in this paper.  The single infinite distributed time delay is considered in~\cite{alber2010} for the 
model of immune system--tumour interactions, however in that paper only an exponential distribution is considered and the 
linear chain trick that reduces the system with infinite delay to a~larger system of ODEs is used.

A~part of our results is based on application of the Mikhailov Criterion, which generalised version is formulated and 
proved in the Appendix. This generalisation plays an essential role in studying local stability of  the steady state
with certain distributed delays. 

The paper is organised as follows: in Section~\ref{sec:model} considered family of angiogenesis models with distributed 
delays is proposed, and a~proper phase space and initial conditions are defined.
In Section~\ref{sec:matprop} the mathematical analysis including the global existence and 
uniqueness of solutions, the stability of the existing steady state of considered family of models for different types of 
considered delay distributions is presented. In that section we also discuss possibilities of stability changes 
of the positive steady state. Next, our stability results are illustrated and extended by 
numerical simulations. In Section~\ref{sec:dis} we discuss and summary our results. Finally,  for completeness 
in~\ref{appen} we formulate generalized Mikhailov Criterion used in the paper and we prove it. 

\section{Family of angiogenesis models with distributed delays}\label{sec:model}

The  effect of the tumour stimulation of the vessels growth as well as the positive influence of the 
vessel network on the tumour dynamics is neither instantaneous nor delayed by
some constant value. Hence, to reflect reality better, instead of constant delays considered earlier 
in~\cite{donofrio2004}, \cite{Piotrowska2011} or~\cite{Forys2013} we consider the delays that are 
distributed around some mean value and their distributions are  given by the general probability densities $f_i$.
We study the following system
\begin{subequations}\label{modeldis}
\begin{align}
	\frac{\dd}{\dd t}p(t)&=rp(t)h\left(\int^{\infty}_0f_1(\tau)\frac{p(t-\tau)}{q(t-\tau)}\dd\tau\right),
								\label{modeldis:1}\\
	\frac{\dd}{\dd t}q(t)&=q(t)\Bigg(b\int^{\infty}_{0}
				f_2(\tau)\left(\frac{p(t-\tau)}{q(t-\tau)}\right)^{\alpha}\dd\tau -a_Hp^{2/3}(t) - \mu \Bigg) 
						,\label{modeldis:2} 
\end{align}
\end{subequations}
where  $f_i(s):[0,\infty)\to \R_{\ge}$ are delay distributions with the 
following properties: 
\begin{enumerate}[\bf (H1)]
	\item $\displaystyle\int^{\infty}_{0}f_i(s)\dd s=1$, $i=1,2$ and 
	\item $\displaystyle 0\le \int^{\infty}_{0}sf_i(s)\dd s<\infty$, $i=1,2$.
\end{enumerate} 
Moreover, we assume that the function $h$ has the following properties:
\begin{enumerate}[\bf (P1)]
	\item $h:(0,\infty) \to \R$ is a~continuously differentiable decreasing function;
	\item $h(1) = 0$;
	\item $h'(1) = -1$.
\end{enumerate} 
Note, that properties \textbf{(P2)} and \textbf{(P3)} do not make our studies less general as a~proper rescaling 
and a~suitable choice of parameter $r$ always allows us to arrive to the case $h(1) = 0$ and $h'(1)=-1$. 
Note also, that an arbitrary probability distribution defined on $[0,+\infty)$ with a~finite expectation fulfils assumptions \textbf{(H1)} and \textbf{(H2)}. 

To close model~\eqref{modeldis} we need to define initial conditions. We consider a~
continuous initial function $\phi:(-\infty,0]\rightarrow\R^2$. For infinite delays we need to regulate the 
behaviour of this function as $t$ tends to $-\infty$. To this end, we introduce a~suitable space. 
Let us denote as $\Cb= \mathbf{C}((-\infty,0],\R^2)$ the space of continuous functions defined on the interval 
$(-\infty,0]$ with values in $\R^2$. For an arbitrary chosen non-decreasing continuous function 
$\eta:(-\infty,0]\to\R^+$ such that $\displaystyle\lim_{\theta\to-\infty}\eta(\theta)=0$,  we define the Banach space  
\[
	\Kb_{\eta}  = \left\{ \varphi \in \Cb : \lim_{\theta\to-\infty} \varphi(\theta)\eta(\theta) = 0\quad 
		\text{ and } \quad \sup_{\theta \in (-\infty,0]} |\varphi(\theta)\eta(\theta)|<\infty\right\},
\]
with a~norm 
\[
	\|\varphi\|_{\eta} = \sup_{\theta \in (-\infty,0]} |\varphi(\theta)\eta(\theta)|,
	\quad \text{ for any } \quad 
	\varphi \in \Kb_{\eta}.
\]
We need initial functions $\phi$ to be in $\Kb_{\eta}$ for some arbitrary non-decreasing continuous 
function $\eta$ that tends to 0 for arguments tending to $-\infty$ (see~\cite{Hino1991}).  
If the delay is finite, that is if supports of $f_1$ and $f_2$ 
are compact, we suppose initial functions $\phi\in \mathbf{C}([-\tau_{\max{}},0],\R^2)$, where the interval
$[-\tau_{\max{}},0]$ contains supports of $f_i$, $i=1,2$. However, this is equivalent to consideration of 
the space $\Kb_{\eta}$ with an arbitrary $\eta$ being equal to $1$ on $[-\tau_{\max{}},0]$ and decreasing to $0$ 
in~$-\infty$. 
If the support of the probability density is unbounded and the initial function $\phi$ is also unbounded, 
than we need to choose an appropriate function $\eta$ (which in turn implies a~choice of the 
phase space $\Kb$). The function $\eta$ must be chosen to control the behaviour of the initial function
in $-\infty$.
However, due to biological interpretation of the model parameter, it is reasonable to restrict our analysis to
the globally bounded initial functions. Thus, we can chose  an arbitrary
positive continuous non-decreasing  function $\eta$ such that $\eta(\theta)\rightarrow 0$ for 
$\theta\rightarrow-\infty$ for example $\eta=e^{\theta}$.  Nevertheless, 
because such assumption would not make theorems' formulation simpler, we decided to 
present and prove theorems for arbitrary initial functions and not to assume a~particular form
of the function $\eta$.

\section{Mathematical analysis of family of angiogenesis models with distributed delays}\label{sec:matprop}

Due to 
properties \textbf{(P2)} and \textbf{(H1)} the steady states of 
system~\eqref{modeldis} are the same as for Hahnfeldt~\etal and d'Onofrio-Gandolfi models without delays or with 
discrete delays. Thus, system~\eqref{modeldis} has a~unique positive steady state 
$(p_e,q_e)$, where $p_e=q_e=\left(\frac{b-\mu}{a_H}\right)^{\frac{3}{2}}$ (compare e.g.~\cite{Piotrowska2011}) if
and only if  $b>\mu$.
Hence, in the rest of the paper we assume that $b>\mu$ holds. 

Note, that solution to the equations defined by~\eqref{modeldis} can be written in the exponential form. Thus, 
using the same argument as in~\cite{onf2009} or in~\cite{Piotrowska2011}  (for the discrete delay cases), 
we deduce that $\R_+^2$ is the invariant set for system~\eqref{modeldis}.

Note, that in system~\eqref{modeldis} terms with delay are of the form $p/q$. This, together with 
the invariance of $\R_+^2$ suggests the following change of variables
\[
x=\ln\left(\frac{p}{p_e}\right), \quad\quad y=\ln\left(\frac{pq_e}{qp_e}\right)
\]
giving
\begin{equation}\label{model_resc_suma}
	\begin{aligned}
		\frac{\dd}{\dd t}x(t)&=rh\left( \int^{\infty}_{0} f_1(\tau)\e^{y(t-\tau)}\dd\tau\right),
								\\
		\frac{\dd}{\dd t}y(t)&=rh\left( \int^{\infty}_{0} f_1(\tau)\e^{y(t-\tau)}\dd\tau\right) 
						-b\int^{\infty}_{0}f_2(\tau)\e^{\alpha y(t-\tau)}\dd\tau+(b-\mu)
						\e^{\frac{2}{3}x(t)}+\mu.
	\end{aligned}
\end{equation}
Newly introduced variable $y$ allows us to consider the system where only one variable ($y$) has delayed 
argument, whereas in system~\eqref{modeldis} both variables have delayed arguments.
 Clearly, the steady state for re-scaled 
system~\eqref{model_resc_suma} is $(x_e,y_e)=(0,0)$.

It should be mentioned here that the scaling procedure transforms  $\R_+^2$ into the whole $\R^2$, so 
space $\Kb_{\eta}$ is the appropriate phase space for the model~\eqref{model_resc_suma}.

\subsection{Existence and uniqueness}
\begin{prop}\label{thm:ex}
	Let $\eta:(-\infty,0]\to\R^+$ be a~continuous, non-decreasing function such that
	$\displaystyle\lim_{\theta\to-\infty}\eta(\theta)=0$, let functions $f_i$ fulfil~\textbf{\upshape 
	(H1)}--\textbf{\upshape (H2)}, and let function $h$ fulfils 
	\textbf{\upshape  (P1)}--\textbf{\upshape  (P3)}. For an arbitrary initial function
	$\phi=(\phi_1,\phi_2)\in\Kb_{\eta}$  there exists $t_{\phi}>0$ such that system~\eqref{model_resc_suma}
	with initial condition $x(t)=\phi_1(t)$, $y(t)=\phi_2(t)$ for $t\in (-\infty,0]$, has a~unique solution in 
	$\Kb_{\eta}$	defined on $t\in[0, t_{\phi})$.
\end{prop}
\begin{proof}
The right-hand side of system~\eqref{model_resc_suma} fulfils a~local Lipschitz condition. 
In fact, the assumption \textbf{(P1)} implies that the derivative of $h$ is bounded on an arbitrarily chosen compact set of $(0,+\infty)$. Thus, 
the function $h$ is locally Lipchitz continuous function as well as all functions on the right-hand side of~\eqref{model_resc_suma}. 
This implies that the Lipschitz condition is fulfilled for any bounded set in $\Kb_\eta$.
Hence, 
there exits a~unique solution to system~\eqref{model_resc_suma} defined on $t\in[0, t_{\phi})$ 
(see~\cite[Chapter 2, Theorem~1.2]{Hino1991}).
\end{proof}

\begin{thm}[global existence]
If assumptions of Proposition~\ref{thm:ex} are fulfilled, the probability densities $f_i$ and an initial function
	are globally bounded, then solutions to~\eqref{model_resc_suma} are 
	defined for all $t\ge 0$. 
\end{thm}
\begin{proof}
	The right-hand side of~\eqref{model_resc_suma} can be written in a~functional form as 
	\[
		\frac{\dd }{\dd t} x(t) = G_1(x_t,y_t)\,, \quad \frac{\dd }{\dd t} y(t) = G_2(x_t,y_t)\,,
	\]
	with
	\[
		\begin{split}
			G_1(\phi_x,\phi_y)&=r h\left( \int^{\infty}_{0} f_1(\tau)\e^{\phi_y(-\tau)}\dd\tau\right),
									\\
			G_2(\phi_x,\phi_y)&=r h\left( \int^{\infty}_{0} f_1(\tau)\e^{\phi_y(-\tau)}\dd\tau\right) 
							-b\int^{\infty}_{0}f_2(\tau)\e^{\alpha \phi_y(-\tau)}\dd\tau+(b-\mu)
							\e^{\frac{2}{3}\phi_x(0)}+\mu,
		\end{split}
	\]
	for $(\phi_x,\phi_y)\in \Cb$.
	Note that due to \textbf{(P1)} we have $|G_1(\phi_x,\phi_y)| \le r h(\exp(-||\phi_y||))$ and similar inequality 
	holds for $|G_2(\phi_x,\phi_y)|$. This implies, that 
	the function $(G_1,G_2)$ maps bounded sets of $\Cb$ into bounded sets of $\R^2$ (so 
	their closures are compact). Thus, if the solution to~\eqref{model_resc_suma} cannot be 
	prolonged beyond the interval $[0,T)$, for some finite $T$, then $\displaystyle\lim_{t\to T} 
	\bigl(\| x_t\|_{\eta} + \|y_t\|_{\eta}\bigr) = +\infty$  (see~\cite[Chapter~2, Theorem~2.7]{Hino1991}). In the following, we show that  $x$ and $y$ are 
	bounded on $[0,T)$ hence, the solution exists for all $t \ge 0$.
	
	Let $\delta>0$ be an arbitrary number such that $\int_{\delta}^{\infty} f_i(\tau)\dd\tau>0$
	for $i=1,2$. Using the step method we choose the time step equal to $\delta$ and we show that the solution 
	to~\eqref{model_resc_suma} can be prolonged on the interval $[0,\delta]$ and hence on the 
	interval  $[n\delta,(n+1)\delta]$, for any $n\in\N$. Let~letters $C_i$ denote constants 
	that will be chosen later in a~suitable way.

	First, we provide an upper bound of the solution. Due to boundedness of $y(t)$ for $t\le 0$, we write
	\[
		\int^{\infty}_{0} f_1(\tau)\e^{y(t-\tau)}\dd\tau = \int^{\delta}_{0} f_1(\tau)\e^{y(t-\tau)}\dd\tau +
		\int^{\infty}_{\delta} f_1(\tau)\e^{y(t-\tau)}\dd\tau \ge \int^{\infty}_{\delta} f_1(\tau)\e^{y(t-\tau)}\dd\tau
		\ge C_1,
	\]
	for all $t\in [0,\delta]$. This estimation leads to a~conclusion
	\[
		\frac{\dd }{\dd t} x(t) \le r\, h\bigl(C_1\bigr) \; \Longrightarrow \; x(t) \le C_2 \quad \text{ for } 
			t\in [0,\delta].  
	\]
	From the second equation of~\eqref{model_resc_suma}, arguing in a~similar way,  we have 
	\[
		\frac{\dd }{\dd t} y(t) \le  r\,h\bigl(C_1\bigr) +(b-\mu)\e^{C_3} +\mu  \; \Longrightarrow \; y(t) \le C_4,
			\quad \text{ for } t\in [0,\delta].  
	\]
	Now, we proceed with a~lower bound of the solutions. 
	Splitting integral on the interval $(0,+\infty)$ into two integrals and using the fact that 
	$y(t)$ is bounded we obtain 
	\[
		\begin{split}
			\int_0^{\infty}f_1(\tau) \e^{y(t-\tau)}\dd\tau &=
				\int_0^{\delta}f_1(\tau) \e^{y(t-\tau)}\dd\tau  + \int_{\delta}^{\infty} f_1(\tau)\e^{y(t-\tau)}\dd\tau
				\le \int_0^{\delta}f_1(\tau) \e^{y_M}\dd\tau  + C_5\\
				& \le \int_0^{\infty}f_1(\tau) \e^{y_M}\dd\tau  + C_5 \le \e^{y_M}+C_5 = C_6\,,
		\end{split}
	\]
	where $y_M$ is an upper bound of $y(t)$ on the interval  $[-\tau,\delta-\tau]$. This estimation together with 
	a similar argument applied to the second integral in the second equation of~\eqref{model_resc_suma} yields 
	\[
		\int_0^{\infty}f_2(\tau) \e^{\alpha y(t-\tau)}\dd\tau \le C_7\,.
	\]
	Therefore, from the second equation of~\eqref{model_resc_suma} we have 
	\[
		\frac{\dd }{\dd t} y(t) \ge r h\bigl(C_6\bigr) - b\, C_7\; \Longrightarrow \;
				y(t) \ge C_8, \; \text{ for } \; t\in [0,\delta]. 
	\]
	The boundedness of $y$ together with the form of first equation of~\eqref{model_resc_suma} implies boundedness 
	of $x(t)$ on $[0,\delta]$.
	Now, knowing that $x(t)$ and $y(t)$ are bounded on the interval $(-\infty,\delta]$, an analogous 
	reasoning allows us to deduce the boundedness of the solution on the interval $(-\infty,2\delta]$.
	 Hence, the mathematical induction yields the boundedness of the solutions 
	to~\eqref{model_resc_suma} 
	on each compact interval included in $[0,+\infty)$, which completes the proof.
\end{proof}

\subsection{Stability and Hopf bifurcations}
In this section we study the local stability of the steady state (0,0) using standard linearisation 
technique.
Linearisation of system~\eqref{model_resc_suma} around the steady state $(0,0)$ has the following form
\begin{equation}\label{sys:lin}
	\begin{aligned}
		\frac{\dd}{\dd t} x(t) &= r h'(1) \int_0^{\infty} f_1(\tau) y(t-\tau)\dd \tau \, ,\\
		\frac{\dd}{\dd t} y(t) &= \frac{2}{3}(b-\mu)x(t) + \int_0^{\infty}\Bigl(r h'(1) f_1(\tau)
					-\alpha b  f_2(\tau)\Bigl) y(t-\tau)\dd \tau \, ,
	\end{aligned}
\end{equation}
and, due to the equality $h'(1)=-1$ (see Property \textbf{(P3)}), the corresponding characteristic function is given by
\[
 W(\lambda) = \det \begin{bmatrix}
			\displaystyle\lambda &\displaystyle r \int_0^{\infty} f_1(\tau) \e^{-\lambda\tau}\dd \tau \\
			\displaystyle-\frac{2}{3}(b-\mu) &
				\displaystyle\lambda+ \int_0^{\infty}\Bigl(rf_1(\tau)+\alpha b  f_2(\tau)\Bigl) \e^{-\lambda\tau}\dd \tau
	\end{bmatrix}.
\]
Thus,
\[
	W(\lambda) = \lambda^2 +\lambda \int_0^{\infty}\Bigl(rf_1(\tau)+\alpha b  f_2(\tau)\Bigl) \e^{-\lambda\tau}\dd \tau
			+ \frac{2r}{3}(b-\mu)\int_0^{\infty} f_1(\tau) \e^{-\lambda\tau}\dd \tau.
\]
In general, as the probability densities one considers the specific distributions which describe the experimental data or
 studied phenomena  in  the best way. In the present paper, we consider two particular types of the probability densities. One type is given by
\begin{equation}\label{eq:zabek}
		f_i(\tau) = \frac{1}{\varepsilon{^2}}\begin{cases}
					\varepsilon -\sigma+\tau, & \tau \in [\sigma-\varepsilon,\sigma), \\
					\varepsilon +\sigma -\tau, & \tau\in [\sigma,\sigma+\varepsilon], \quad\quad \textrm{for} \quad\sigma\geq\varepsilon, i=1,2, \\
					0 & \text{otherwise}.
				\end{cases}
\end{equation}
We call probability densities defined by~\eqref{eq:zabek} piecewise linear probability densities.
Note, that for $\varepsilon\to 0$ the functions $f_i$ converge to Dirac delta at the points 
$\sigma$ and therefore, system~\eqref{modeldis} becomes a~system with discrete delays
considered by~\cite{donofrio2004,mbuf09appl,Piotrowska2011,Forys2013}.
Since we assume that all considered probability densities are defined on interval $[0,\infty)$ condition 
$\sigma\ge\varepsilon$ must be fulfilled. Note, that for 
$\sigma\ge\varepsilon$ the average value of $f_i$ is equal to $\sigma$ and the standard deviation is equal to 
$\varepsilon/\sqrt{6}$. For $\sigma<\varepsilon$ one obtains so-called neutral equations, for details see~\cite{Hale93}.

In this paper, we also consider second type of probability densities so-called the Erlang probability densities separated or not from zero 
by~$\sigma\geq0$, i.e. we study system~\eqref{modeldis} with functions given by
\begin{equation}\label{eq:Erlang}
	f_1(\tau) = g_{m_1}(\tau-\sigma)\,, \qquad
	f_2(\tau) = g_{m_2}(\tau-\sigma)\,,
\end{equation}
where $g_{m_i}(\tau)$, $i=1,2$,  are called non-shifted Erlang distributions and are defined by 
\begin{equation}\label{eq:defE}
g_{m_i}(s)=\frac{a}{(m_i-1)!}(as)^{m_i-1}\e^{-as},
\end{equation}
with $a>0$, $s\geq 0$. The mean value of the~non-shifted Erlang distribution $g_{m_i}$ is given 
by $\frac{m_i}{a}$, while the variance is equal to $\frac{m_i}{a^2}$. Hence, 
the average delay is equal to this mean and the deviation $\sqrt{\frac{m_i}{a^2}}$ measures the 
degree of concentration of the delays about the average delay. Clearly, the non-shifted Erlang 
distribution is a~special case of the Gamma distribution, where the shape parameter $m_i$ is an 
integer. It is also easy to see that the non-shifted Erlang distribution is the generalisation of 
the exponential distribution, which one gets taking $m_i=1$. On the other hand, for the non-shifted Erlang distributions when 
$m_i\rightarrow+\infty$ the probability densities $g_{m_i}$ converge to a~Dirac distributions and hence the 
system with discrete delays~\eqref{model} is recovered as a~limit of Erlang distributions for 
system~\eqref{modeldis}. For the shifted Erlang distribution the mean value is 
$\sigma+\frac{m}{a}$, while variances stay the same as for the non-shifted case. 

\subsubsection{Erlang probability densities}

The Erlang probability densities separated from zero by $\sigma$  read
\[
	f_i(\tau) = \frac{a^{m_i}(\tau-\sigma)^{m_i-1}}{(m_i-1)!}\e^{-a(\tau-\sigma)},\;\; \text{ for } \; \tau\ge \sigma
\]
and $0$ otherwise, for $i=1,2$. Note, that in this paper we consider only the case $a>0$. Then
for probability densities given by~\eqref{eq:Erlang}--\eqref{eq:defE} we have
\[
	\int_0^{\infty} f_i(\tau) \e^{-\lambda\tau}\dd \tau=  \frac{a^{m_i}}{(a+\lambda)^{m_i}} \e^{-\lambda\sigma}.
\]
Hence, the characteristic function has the following form
\begin{equation}\label{eq:erl:dow}
	W(\lambda) = \lambda^2 +\lambda\left(r\frac{a^{m_1}}{(a+\lambda)^{m_1}}\e^{-\lambda\sigma}+\alpha b\frac{a^{m_2}}{(a+\lambda)^{m_2}}\e^{-\lambda\sigma} 
\right)
+ \frac{2r}{3}(b-\mu) \frac{a^{m_1}}{(a+\lambda)^{m_1}}\e^{-\lambda\sigma}.
\end{equation}
Define
\begin{equation}\label{def:betagamag}
	\beta = r+\alpha b \quad \text{and} \quad
	\gamma = \frac{2 r\, (b-\mu)}{3}.
\end{equation} 

\begin{prop}\label{thm:non-shiftE}
Let $b>\mu$. The trivial steady state of system~\eqref{model_resc_suma} with the non-shifted Erlang probability density 
 given by~\eqref{eq:Erlang}--\eqref{eq:defE} ($\sigma=0$) is locally asymptotically stable if

\begin{enumerate}[\upshape (i)]
 	\item $a\beta>\gamma$ for $m_1=m_2=1$ ;
	\item $a >  \frac{1}{2}\beta + \frac{2\gamma}{\beta}$ for $m_1=m_2=2$; 
	\item $a > \frac{9}{8}\beta$ and $8\beta a^3-3(8\gamma+3\beta^2)a^2+3\gamma\beta a-\gamma^2>0$ for $m_1=m_2=3$;
	\item $2a(a+r)>\Big(a(a+\alpha b)+\gamma\Big)+\frac{4a^2\gamma}{(a(a+\alpha b)+\gamma)}$ for $m_1=1$ and $m_2=2$; 
	\item $a > \frac{1}{2}\beta+ \frac{2\gamma}{\beta}-\alpha b$ for $m_1=2$ and $m_2=1$.
\end{enumerate}
\end{prop}
\begin{proof}
For $\sigma=0$ we clearly have
\[
	W(\lambda) = \lambda^2 +\lambda\left(r\frac{a^{m_1}}{(a+\lambda)^{m_1}}+\alpha b\frac{a^{m_2}}{(a+\lambda)^{m_2}}
\right)
+ \frac{2r}{3}(b-\mu) \frac{a^{m_1}}{(a+\lambda)^{m_1}}\,.
\]
As it can be seen, the advantage of using the Erlang distributions (not separated from zero) is that instead of studying existence of
the zeros of  the characteristic function one can
 study existence of roots of a~polynomial, 
thus the stability analysis is easier.

First, consider $m_1=m_2=m$. We investigate the behaviour of the roots of polynomial
\begin{equation}\label{eq:Erl:m=1}
	 \lambda^2(a+\lambda)^m + a^m (\lambda\,\beta + \gamma),
\end{equation}
where $\beta$ and $\gamma$ are given by~\eqref{def:betagamag}.

The Routh-Hurwitz stability criterion (see eg.~\cite{rh-criterion}) gives that the necessary and sufficient conditions 
for the stability of trivial steady state of system~\eqref{sys:lin}. 
Clearly, the degree of polynomial~\eqref{eq:Erl:m=1} is larger for larger $m$. However, in each case of equals 
$m$'s (sometimes tedious) algebraic calculations lead to some conditions  that guarantee the stability of the considered steady state.

For $m=1$ (i.e. for the exponential distribution not separated from zero) we have 
a~polynomial of degree three
\[ \lambda^2(a+\lambda) + a~(\lambda\,\beta + \gamma)\,,\]
 while for $m=2$  we have 
\[
\lambda^4+2a\,\lambda^3+a^2\,\lambda^2+\beta a^2\, \lambda+\gamma a^2\,.
\]
The case $m=3$, is the most complicated since a~direct calculation shows that 
\eqref{eq:Erl:m=1} 
is the polynomial of degree 5
\[
\lambda ^5 +3 a~\lambda^4+3 a^2 \lambda ^3  +a^3 \lambda ^2+a^3 \beta  \lambda+a^3 \gamma\,.
\]
For $m_1=1$ and $m_2=2$, polynomial~\eqref{eq:Erl:m=1} reads 
\[
	\lambda^4 + 2a\,\lambda^3+a(a+r)\,\lambda^2+\Bigl(a^2\beta+\gamma a\Bigr)\,\lambda+
		\gamma a^2\,,
\]
while for $m_1=2$ and $m_2=1$ we have
\[
	\lambda^4+ 2a\,\lambda^3 + a(a+\alpha b)\,\lambda^2+a^2\beta\,\lambda +\gamma a^2\,.
\]
\end{proof}

\begin{prop}
	For $b>\mu$, $\sigma=0$, $m_1=1$, $m_2=2$ and 
	\begin{equation}\label{cond:prop}
	r<\alpha b \quad \text{ or } \quad r > \frac{\alpha b}{2} + \frac{2\gamma}{\alpha b},
	\end{equation}
	there exists $\bar a>0$ such that for $a>\bar a$ the trivial steady state of system~\eqref{model_resc_suma} with the
	non-shifted Erlang probability densities given by~\eqref{eq:Erlang}--\eqref{eq:defE}  is locally asymptotically 
	stable and it is unstable for $a\in(0,\bar a)$. 
\end{prop}

\begin{proof}
Since $a(a+\alpha b)+\gamma>0$ holds, the stability condition for the case $m_1=1$, $m_2=2$ 
(the condition (ii) of Theorem~\ref{thm:non-shiftE}) is equivalent to
\[
W_a(a) = a^4+2r\, a^3+\Bigl(\alpha b(2r-\alpha b)-4\gamma\Bigr)\,a^2 +2\gamma(r-\alpha b)\,a - \gamma^2>0. 
\]
The Descart's rule of signs implies that if one of conditions~\eqref{cond:prop} holds,  
then there exists exactly one simple positive real root of $W_a(a)$. We denote it by $\bar a$.  Then
the trivial steady state system~\eqref{model_resc_suma}  is locally asymptotically for $a>\bar a$ and unstable for $0<a<\bar a$. 
\end{proof}

\begin{thm}\label{thm:Erlang}
Considered system~\eqref{model_resc_suma} with the shifted Erlang probability densities given by~\eqref{eq:Erlang}--\eqref{eq:defE}. 

\begin{enumerate}[\upshape (a)]
	\item Let $m_i=m$, $i=1,2$. If the trivial steady state is
	\begin{enumerate}[\upshape (i)]
		\item unstable for $\sigma=0$,  then it is unstable for any $\sigma>0$,
		\item locally asymptotically stable for $\sigma=0$, then there exits $\sigma_0>0$ such that 
		it is locally stable for all $\sigma\in[0,\sigma_0)$ and unstable for $\sigma>\sigma_0$.
	 	At $\sigma=\sigma_0$ the Hopf bifurcation occurs.
	\end{enumerate}
	\item Let $m_1=1$, $m_2=2$.  If the trivial steady state is locally asymptotically stable for $\sigma=0$, 
	then there exists exactly one $\sigma_0>0$ such that  for $\sigma \in [0,\sigma_0)$ it is locally asymptotically stable and it is unstable for $\sigma>\sigma_0$. 
	At $\sigma=\sigma_0$ the Hopf bifurcation occurs.
	\item Let $m_1=2$, $m_2=1$.
	\begin{enumerate}[\upshape (i)]
		\item If the steady state is locally asymptotically stable for $\sigma=0$, 
			and the function 
				\begin{equation}\label{eq:aux:sigma3}
					F(u)=u^4 +2a^2u^3+ u^2a^2\bigl(a^2-\alpha^2 b^2\bigr)-
						u a^3\Biggl( a\beta^2 - 2\alpha b\gamma\Biggr)
						- \gamma^2 a^{4}.
				\end{equation}
		has no positive multiple roots, 
		 then the steady state is locally asymptotically stable for some $\sigma \in [0,\sigma_0)$
				and it is unstable for $\sigma>\bar\sigma$, with $\sigma_0\le \bar\sigma$. 
				At $\sigma=\sigma_0$ the Hopf bifurcation occurs.
		\item If
				\begin{equation}\label{cond:stab:m11m22}
					a \ge \alpha b\min\left\{1,\frac{2 \gamma}{\beta^2}\right\},
				\end{equation}
				 or 
				\begin{equation}\label{cond:stab:m11m22-verA}
					a < \alpha b\min\left\{1,\frac{2 \gamma}{\beta^2}\right\}
					\; \text{ and } \; 
					\Bigl(a^2+2\alpha^2 b^2\Bigr)^3 <27\Bigl(2\alpha b\gamma+ a\bigl(\alpha^2b^2-\beta^2\bigr)\Bigr)^2,
				\end{equation}
				then at most one stability switch of the steady state is possible.  Moreover, if the 	
				steady state is locally asymptotically
				stable for $\sigma=0$, then it is stable for some $\sigma \in [0,\sigma_0)$
				and it is unstable for $\sigma>\sigma_0$, 
			 For $\sigma=\sigma_0$ the Hopf bifurcation is observed.
	\end{enumerate}
\end{enumerate}
\end{thm}
\begin{proof}
	Let $m_M = \max\{m_1,m_2\}$. Then the characteristic function~\eqref{eq:erl:dow} has the following form 
	\[
		\begin{split}
		W(\lambda) =& \frac{1}{(a+\lambda)^{m_M}}\Biggl(
				\lambda^2(a+\lambda)^{m_M} +\\ &\quad  + \biggl(\lambda(ra^{m_1}(a+\lambda)^{m_M-m_1} 
				+\alpha ba^{m_2}(a+\lambda)^{m_M-m_2})
				+ \gamma a^{m_1}(a+\lambda)^{m_M-m_1}\biggr)\e^{-\lambda\sigma}\Biggr).
		\end{split}
	\]
	The zeros of $W(\lambda)$ are the same as the zeros of 
	\begin{equation}\label{eq:chr:max}
		D(\lambda) = \lambda^2(a+\lambda)^{m_M} + \biggl(\lambda(ra^{m_1}(a+\lambda)^{m_M-m_1} 
				+\alpha ba^{m_2}(a+\lambda)^{m_M-m_2})
				+ \gamma a^{m_1}(a+\lambda)^{m_M-m_1}\biggr)\e^{-\lambda\sigma}.
	\end{equation}
	For $\lambda=i\omega$ define an auxiliary function
	\begin{equation}\label{eq:aux:ogolnie}
	 F(\omega) = \omega^4(a^2+\omega^2)^{m_M} - 
		\left\|i\omega(ra^{m_1}(a+i\omega)^{m_M-m_1} 
				+\alpha ba^{m_2}(a+i\omega)^{m_M-m_2})
				+ \gamma a^{m_1}(a+i\omega)^{m_M-m_1} \right\|^2.
	\end{equation}
	In the following we show, with one exception, that there exists a~unique 
	single positive zero
	 $\omega_0=\sqrt{u_0}$ of the function $F$ such that $F'(\omega_0)>0$.
	This, together with  Theorem~1 from~\cite{cook86ekvacioj} allows us to deduce that 
	the zeros of $D$ crosses imaginary axes from left to right with a~positive velocity. 
	This yields the existence of stability switches and the occurrence of the Hopf bifurcation
	in the appropriate cases.
	
	For $m_1=m_2=m$, and for $u=\omega^2$ the auxiliary function $F$ takes the form
	\begin{equation}\label{eq:Fm}
	F(u)=u^2\Bigl(a^2+u\Bigr)^{m}-a^{2m}(u\beta^2+\gamma^2)\,.
	\end{equation}
	We are interested in the existence of the real positive roots of~\eqref{eq:Fm}.
 Because the number of sign changes 
	between consecutive non-zero coefficients is equal to 1
the Descartes' rule of signs indicates that the polynomial~\eqref{eq:Fm} has one simple real positive root 
	denoted by $u_0$. Clearly, the fact that coefficient of $u^{2+m}$ is positive and $F(0)<0$, implies $F'(u_0)\ge 0$. 

	Moreover, 
	\[
	F'(u)=2u\Bigl(a^2+u\Bigr)^{m}+ mu^2(a^2+u)^{m-1}-a^{2m}\beta^2
	\]
	and thus, $F''(u)>0$.  Hence, $F'(u_0)>0$. This completes the proof of part (a).

	For $m_1=1$, $m_2=2$ and $u=\omega^2$ the function $F$ given by \eqref{eq:aux:ogolnie}  has the form
	\begin{equation}\label{eq:aux:nierowne}
				 F(u) =  u^4 
			+ 2 a^2\, u^3 
			+ a^2\Bigl(a^2-r^2\Bigr)\,u^2 
			- a^2\left(a^2\beta^2 + 2a~\alpha b \gamma
				+\gamma^2
				\right) \, u
			- 	a^4\gamma^2\,.
	\end{equation}
	We show that the coefficient of $u$ in~\eqref{eq:aux:nierowne}  is negative and hence the Descartes' rule of signs 
	implies that a~single
	stability switch for the trivial steady state of~\eqref{model_resc_suma} occurs. This is 
	equivalent to the inequality 
	\[
		G(a) = a^2\beta^2 + 2a~\alpha b \gamma +\gamma^2>0,
	\]
	for $a\ge 0$. For $b>\mu$ (that is the condition required for the existence of the trivial steady state of~\eqref{model_resc_suma}) 
	the discriminant of $G$ is the following 
	\[
		4\gamma^2\left(\alpha^2 b^2 - \beta^2\right) = 4\gamma^2\left(\alpha^2 b^2 - (r+\alpha b)^2\right)=
		 -4r\gamma^2 \left(r+2\alpha b\right) < 0.
	\]
	Thus, $G(u)>0$, so $F$ has exactly one simple positive root $u_0$. Moreover, 
	$F(0)<0$ hence	$F'(u_0)>0$ and the proof of part (b) is completed. 

	Consider the  case: $m_1=2$ and $m_2=1$. Then for $u=\omega^2$ Eq.~\eqref{eq:aux:ogolnie} takes 
	the form~\eqref{eq:aux:sigma3}
	First, note that for the auxiliary function given by \eqref{eq:aux:sigma3} inequality $F(0)<0$, holds. 
	Thus,  $F$ has at least one real positive zero. Hence, if $F$ has only simple positive zero, by Theorem~1 
	from~\cite{cook86ekvacioj} the steady state 	of system~\eqref{model_resc_suma} is unstable for sufficiently large 
	$\sigma$, which completes the proof of part (c.i).

	Now, we prove statement (c.ii). Assume that~\eqref{cond:stab:m11m22} holds. Then, if $a\ge \alpha b$, then the 
	coefficient of $u^2$ is non-negative and independently on the sign of 
	the coefficient of $u$ the Descartes' rule of signs indicates that there is exactly one simple positive 
	real zero of $F(u)$. Thus, at most one stability switch of the steady state can occur. On the other hand,  
	if condition  $a\ge2\alpha\gamma b/\beta^2$ holds, then the coefficient of $u$ in $F(u)$ is non-positive. Hence, 
	independently of the sign of the coefficient of $u^2$ the single change of the sign is observed. Thus,  $F(u)$ has 
	exactly one simple real positive zero.
	
	Now assume that~\eqref{cond:stab:m11m22-verA} holds. To shorten notation denote
	\begin{equation}\label{def:alpha12}
		\alpha_1 = 2\alpha b\gamma-  a\beta^2, \quad\quad
		\alpha_2 = \alpha^2 b^2-a^2. \quad 
	\end{equation}
	The first inequality of~\eqref{cond:stab:m11m22-verA} is equivalent to $\alpha_1>0$ and $\alpha_2>0$. 
	Denote $F_1(\zeta) = F(a\zeta)/a^4$, with $\zeta=u/a$. Clearly, $F(u)=0$ is equivalent 
	to $F_1(\zeta)=0$. The function $F_1$ reads
	\[
		F_1(\zeta) = \zeta^4 +2a \zeta^3-\alpha_2 \zeta^2 + \zeta \alpha_1 - \gamma^2.
	\]
	We show that under the assumption~\eqref{cond:stab:m11m22-verA} the function $F_1$ is 
	a strictly monotonic function of $\zeta$, which implies that $F$ is a~strictly monotonic function of $u$. 
	This, together with the fact $F(0)<0$ implies that $F$ has exactly one simple positive zero so the assertion 
	of the point (c.ii) is true. 

	In order to show monotonicity of $F_1$, we prove that its first derivative is positive. We have 
	\[
		F_1'(\zeta) = 4\zeta^3 +6a \zeta^2-2\alpha_2 \zeta + \alpha_1.
	\]
	The assumption $\alpha_1>0$ implies $F_1'(0)>0$. Now, we derive a~condition that guarantees 
	the positivity of $F_1'$ for all $\zeta>0$. To this end, we calculate the minimum of $F_1'(\zeta)$ 
	on the 
	interval $[0,+\infty)$. Calculating the second derivative of $F_1$ we find that the minimum of $F_1'$
	is reached at the point 
	\[
		\bar \zeta = \frac{1}{2}\biggl(-a+\sqrt{a^2+\tfrac{2}{3}\alpha_2}\biggr).
	\]
	Using the fact that $F_1''(\bar\zeta) = 0$ we have 
	\[
		F_1'(\bar \zeta) =  \alpha_1 - 2\bar\zeta^3-\alpha_2\bar \zeta.
	\]
	Thus, $F_1$ is strictly increasing for all $\zeta>0$ if and only if 
	\begin{equation}\label{war:sigma3}
		 \bigl(2\bar \zeta^2+\alpha_2\bigr)\bar \zeta< \alpha_1.
	\end{equation}
	A~few algebraic manipulations indicate  that~\eqref{war:sigma3} is equivalent to 
	\[
		\left(a^2+\frac{2}{3}\alpha_2\right)^{3/2} - a\Bigl(a^2+\alpha_2\Bigr)<\alpha_1.
	\]
	Using the definition of~$\alpha_2$ the above inequality reads
	\[
		\frac{1}{3\sqrt{3}}\Bigl(a^2+2\alpha^2 b^2\Bigr)^{3/2} <\alpha_1+ a~\alpha^2b^2.
	\]
	Due to the assumption $\alpha_1>0$ both sides of the above inequality are positive, 
	so squaring that inequality we obtain
	\[
		\Bigl(a^2+2\alpha^2 b^2\Bigr)^3 <27\Bigl(\alpha_1+ a~\alpha^2b^2\Bigr)^2.
	\]
	Now using the definition of $\alpha_1$ we get 
	\[
		\Bigl(a^2+2\alpha^2 b^2\Bigr)^3 <27\Bigl(2\alpha b\gamma+ a\bigl(\alpha^2b^2-\beta^2\bigr)\Bigr)^2,
	\]
	which is fulfilled due to the second inequality of~\eqref{cond:stab:m11m22-verA}. This 
	completes the proof of part (c).
\end{proof}

Note, that there exist a~set of parameters of system~\eqref{model_resc_suma} with the shifted Erlang probability densities given by~\eqref{eq:Erlang}--\eqref{eq:defE} with  
$m_1=2$, $m_2=1$ such that the steady state is locally asymptotically stable for $\sigma=0$ and  conditions 
\eqref{cond:stab:m11m22-verA} hold. As a~proper 
example we formulate the remark below.

\begin{rem}
	If
	\begin{equation}\label{notcontra}
		\alpha b<\beta, \quad \text{ and } \quad 
		1<\frac{2\gamma}{\beta^2},
	\end{equation}
	then for all $a>\frac{1}{2}\beta+\frac{2\gamma}{\beta}-\alpha b$ the steady state  of system~\eqref{model_resc_suma} with the shifted Erlang probability densities given 
	by~\eqref{eq:Erlang}--\eqref{eq:defE} with $m_1=2$, 
$m_2=1$	loses its stability at $\sigma=\bar\sigma$ and the Hopf bifurcation occurs at this point.
\end{rem}
\begin{proof}
	First, note that under the assumptions of the remark if $a\ge\alpha b \min\left\{1,\frac{2\gamma}{\beta^2}\right\}$, then condition
	\eqref{cond:stab:m11m22} holds and Theorem~\ref{thm:Erlang} implies the assertion of the remark. 
	If $a>\frac{1}{2}\beta+\frac{2\gamma}{\beta}-\alpha b$ and conditions \eqref{cond:stab:m11m22-verA}
	hold then Theorem~\ref{thm:Erlang} also implies the assertion of the remark. 
	We show that for some set of parameters these conditions can be fulfilled 
	simultaneously. 

	Consider $a<\alpha b \min\left\{1,\frac{2\gamma}{\beta^2}\right\}$. The second inequality of~\eqref{notcontra} implies that $\min\left\{1,\frac{2\gamma}{\beta^2}\right\}=1$.
	Hence, the stability condition of steady state for $\sigma=0$ and the first condition~\eqref{cond:stab:m11m22-verA} 
	reads
	\begin{equation}\label{zakresa}
		\frac{1}{2}\beta+\frac{2\gamma}{\beta}-\alpha b< a< \alpha b.
	\end{equation}
	First, we show for some parameters inequalities~\eqref{zakresa} give a~non empty set of $a$. This is true 
	if the following inequalities
	\begin{equation}\label{zakresalphab}
		\frac{1}{2}\beta+\frac{2\gamma}{\beta} < 2\alpha b< 2\beta
	\end{equation} 
	hold, where the second inequality in~\eqref{zakresalphab} is just the first one in~\eqref{notcontra}. 
	Inequalities~\eqref{zakresalphab} do not contradict each other
	if and only if
	\begin{equation}\label{zakresgamma}
		\frac{1}{2}\beta+\frac{2\gamma}{\beta} < 2\beta.
	\end{equation} 
	Thus, if $\frac{2\gamma}{\beta^2}<3/2$, then there is a~non-empty set of $\alpha b$ such that 
	inequalities~\eqref{zakresalphab} hold and for such $\alpha b$ 
	\eqref{zakresa} determines a~non-empty set of $a$.
	
	Now, we show that if~\eqref{notcontra} and the first inequality of~\eqref{cond:stab:m11m22-verA}
	hold then the second inequality of~\eqref{cond:stab:m11m22-verA} also holds. 
	
	Note that the assumption $\alpha b<\beta$  means that the right hand side of the second inequality 
	of~\eqref{cond:stab:m11m22-verA} is a~decreasing function of $a$ (of course we need to use here the first inequality 
	of~\eqref{cond:stab:m11m22-verA} that guarantees the  positivity of the expression under the second power), 
	while the left hand-side is an increasing function of $a$. 
	Since our assumptions imply that $a<\alpha b$, it is enough to check if 
	inequality~\eqref{cond:stab:m11m22-verA} holds for  $a=\alpha b$.

	Rearranging terms we obtain 
	\[
		27\alpha^6 b^6 <27\Bigl(\alpha^3b^3 + \alpha b\bigl(2\gamma-\beta^2\bigr)\Bigr)^2.
	\]
	The above inequality obviously holds as $2\gamma>\beta^2$, which completes the proof.
\end{proof}

For the	 general case of system~\eqref{model_resc_suma} with the shifted Erlang probability densities given 
by~\eqref{eq:Erlang}--\eqref{eq:defE}  the characteristic  function $D(\lambda)$ is given by~\eqref{eq:chr:max} and the auxiliary function 
$F$ is given by~\eqref{eq:aux:ogolnie}. It is easy to see that the highest power of $\omega$ is $4+2m_M$ and 
the coefficient of it is $1$. Moreover, we have $F(0)<0$.  Thus, there exists $\omega_0>0$ such that 
$F(\omega_0)=0$ and the roots cross the imaginary axis from the left to the right half-plane. 
Thus, using Theorem~1 from~\cite{cook86ekvacioj}, we can deduce the following

\begin{rem}
	If all roots of $F(\omega)$ given by~\eqref{eq:aux:ogolnie} are simple, and 
	the trivial steady state of system~\eqref{model_resc_suma} with the shifted Erlang probability densities given by~\eqref{eq:Erlang}--\eqref{eq:defE} is locally asymptotically stable for $\sigma =0$, then it loses its stability due to the Hopf bifurcation for some $\sigma_0>0$ 
	and it is unstable for $\sigma>\sigma_{\infty}$ with some $\sigma_{\infty}\ge \sigma_0$. 
\end{rem}

It seems that the case of multiple roots of $F(\omega)$ is non-generic, however we will not prove it here.

\subsubsection{Piecewise linear probability densities}
For the function defined by~\eqref{eq:zabek} we have 
\bigskip
\[
\int_0^{\infty} f_i(\tau)\e^{-\lambda\tau}\dd \tau = 
	\frac{\e^{-\lambda\sigma}}{\lambda^2\varepsilon^{2}} \left(\e^{\lambda\varepsilon}+\e^{-\lambda\varepsilon}-2\right)
	= \frac{2\e^{-\lambda\sigma}}{\lambda^2\varepsilon^{2}}\Bigl(\cosh(\lambda\varepsilon)-1\Bigr)
\]
and thus the characteristic function for the trivial steady state of system~\eqref{model_resc_suma} reads
\[
	W(\lambda) = \lambda^2 +\lambda \int_0^{\infty}\Bigl(rf_1(\tau)+\alpha b  f_2(\tau)\Bigl) \e^{-\lambda\tau}\dd \tau
			+\gamma\int_0^{\infty} f_1(\tau) \e^{-\lambda\tau}\dd \tau\,,
\]
which is equivalent to
\begin{equation}\label{eq:char:zabek:og}
	W(\lambda) = \lambda^2 +(\lambda r+\gamma)\frac{2\e^{-\lambda\sigma}}{\lambda^2\varepsilon{^2}}\Bigl(\cosh(\lambda\varepsilon)-1\Bigr)+\lambda\alpha b \frac{2\e^{-\lambda\sigma}}{\lambda^2\varepsilon{^2}}\Bigl(\cosh(\lambda\varepsilon)-1\Bigr).
\end{equation}
Finding zeros of function~\eqref{eq:char:zabek:og} is not a~trivial task.
If  the characteristic function has the form $W(\lambda)=P(\lambda)+Q(\lambda)\e^{-\lambda\tau}$, 
where $P$ and $Q$ are polynomials,
one could use the Mikhailov Criterion~\cite{uf04jbs},  to 
estimate   the number of zeros of the characteristic function lying in the right half of complex plane. 
However, in the considered case we do not have strict polynomials. Therefore, we use a~generalised  Mikhailov 
Criterion (for details see~\ref{appen}).  Clearly,  $\cosh(z)=\sum_{n=0}^{\infty}\frac{z^{2n}}{(2n)!}$, for $z\in\C$ 
is an analytic function, hence $\frac{\cosh(z)-1}{z^2}$ is also analytic.
Consequently, the function 
$W$ defined by~\eqref{eq:char:zabek:og} is also analytical. Moreover, obviously condition~\eqref{charfun} holds. Thus, we can apply the generalized  Mikhailov Criterion.

Considering  $\sigma\geq\varepsilon>0$,  we have
\begin{equation}\label{char:zabek:sigmy}
	W(\lambda) = \lambda^2 +\Big(\lambda \beta +
						\gamma\Big)g\bigl(\lambda\varepsilon\bigr)\e^{-\lambda\sigma}\,,
\end{equation}
where  
\begin{equation}\label{def:betagamag1}
	g(x) = \frac{2(\cosh x-1)}{x^2}\,, \quad
	g_1(x) =g(ix) = \frac{2 - 2\cos x}{x^2} .
\end{equation}
Thus,
\begin{equation}\label{reWzabek}
	\Re (W(i\omega)) = -\omega^2 + g_1\bigl(\omega\varepsilon\bigr)\Big(\gamma\cos(\omega\sigma)+
		\beta\omega\sin(\omega\sigma)\Big)\,,
\end{equation}

\begin{equation}\label{imWzabek}
	\Im(W(i\omega)) = g_1\bigl(\omega\varepsilon\bigr)\Big( \beta \omega\cos(\omega\sigma)
			- \gamma\sin(\omega\sigma)\Big) \,.
\end{equation}

Here we formulate sufficient condition for stability of the trivial steady state of system~\eqref{model_resc_suma} 
for the piecewise linear distributions defined by~\eqref{eq:zabek} with $\sigma\geq\varepsilon>0$. This condition is clearly not necessary, what we will show numerically in the next section.

\begin{prop}\label{prop:zabek:stab}
	Let $b>\mu$, $\beta$ and $\gamma$ be defined by~\eqref{def:betagamag}. Then for
\begin{enumerate}[\upshape (i)]
\item	\begin{equation}\label{cond:stab:zabek}
		\sigma< \frac{\pi}{\beta+\sqrt{\beta^2+4\gamma}+\frac{\gamma}{\beta}\pi} 
	\end{equation}
	the trivial steady state of system~\eqref{model_resc_suma} with the piecewise linear distributions defined 
	by~\eqref{eq:zabek} and $\sigma\ge\varepsilon>0$ is locally asymptotically stable.

\item\begin{equation}\label{cond:nstab:zabeka}
	\frac{\beta}{\gamma}<\sigma < \frac{2\pi}{\beta+\sqrt{\beta^2+4\gamma}}
\end{equation}
the trivial steady state of system~\eqref{model_resc_suma} with  the  piecewise linear distributions defined by~\eqref{eq:zabek} 
with $\sigma\geq\varepsilon>0$ is unstable.
\end{enumerate}
\end{prop}
\begin{proof}
	We use the generalized Mikhialov Criterion (for details see~\ref{appen}) to show that the change of the argument of $W(i\omega)$ as $\omega$ varies from 0 to $\infty$
	is equal to $\pi$. In fact, we show that the behaviour of the hodograph is more restrictive, namely 
	that for some $\bar\omega>0$ we have
	$\Re(W(i\omega))<0$ for $\omega>\bar \omega$,  while $\Im(W(i\omega))>0$ 
	for $\omega\in [0,\bar\omega]$. To this end first, note that if $g_1(\omega\epsilon) = 0$ and $\omega>0$, 
	then $\Re(W(i\omega))<0$. Thus, without loss of generality we assume that $g_1(\omega\epsilon) > 0$. 

	First, let us estimate the real part of $W(i\omega)$ given by~\eqref{reWzabek}. Since $b>\mu$ and
	$g_1(\varepsilon\omega)\le 1$ we have 
	\[
		\Re(W(i\omega)) \le -\omega^2 + \gamma +\omega\,\beta =: g_R(\omega).
	\]
	Now,  let us  consider the imaginary part of $W(i\omega)$. In fact, since we are interested in the sign of
	it and $g_1(\varepsilon\omega)\ge 0$ it is enough to consider the sign of
	\[
		\omega\beta\cos(\omega\sigma)- \gamma\sin(\omega\sigma)=g_I(\omega\sigma),
	\]
	where
	\[
		g_I(x) = \frac{\beta}{\sigma}\, x\cos x - \gamma \sin x. 
	\]
	Note, that~\eqref{cond:stab:zabek} implies $\beta>\gamma\sigma$, and thus 
	$g_I(x)>0$ for $x$ close to 0. On the other hand, $g_I(\pi/2)=-\gamma<0$. It is also easy to 
	check, that $g_I(x)$ has a~unique zero in $(0,\pi/2)$ for $\beta>\gamma\sigma$. Below,
	we give some estimation for this zero. 

	Since $\tan x < \frac{\pi}{2} \frac{x}{\frac{\pi}{2}-x}$, for $x\in (0,\pi/2)$ we have 
	\begin{equation}\label{nax0}
		\frac{\beta}{\gamma\sigma} x \ge  \frac{\pi}{2}\frac{x}{\frac{\pi}{2}-x} > \tan x.
	\end{equation}
	A~simple algebraic yields that the first inequality of~\eqref{nax0} is fulfilled for all
	\[
		0< x \le \frac{\pi}{2}\left(1-\sigma\cdot\frac{\gamma}{\beta}\right)<\frac{\pi}{2}.
	\]
	Thus, as long as $\omega\sigma \le  \frac{\pi}{2}\left(1-\sigma\cdot\frac{\gamma}{\beta}\right)$ 
	we have $\Im(W(i\omega))>0$. 

		However, $g_R$ is a~quadratic polynomial with $g_R(0)>0$. Moreover,  the coefficient 
	of $\omega^2$ is negative and the positive root of $g_R$ is  equal to 
	$\frac{1}{2}\left(\beta+\sqrt{\beta^2+4\gamma}\right)$. An easy algebraic 
	manipulation shows that condition~\eqref{cond:stab:zabek} is equivalent to
	\[
		\frac{\sigma}{2}\left(\beta+\sqrt{\beta^2+4\gamma}\right) < \frac{\pi}{2}\left(1-\sigma\cdot\frac{\gamma}{\beta}\right)
	\]
	and thus  $g_R(\omega)<0$ for all $\sigma\omega>\frac{\pi}{2}\left(1-\sigma\cdot\frac{\gamma}{\beta}\right)$,
	which completes the proof of part (i).

	The idea of the proof of part (ii) is similar to the proof of part (i). 
	However, this time we want to show that for $\sigma\omega\in(0,\pi)$ the imaginary part $\Im(W(i\omega))<0$, 
	while for $\sigma\omega>\pi$ the real part $\Re(W(i\omega))<0$. This indicates that the hodograph 
	goes below the $(0,0)$ point on the complex plane when $\omega$ tends to $+\infty$ and hence the change 
	of the argument of $W(i\omega)$ differs from $\pi$. Hence, from the generalized Mikhialov Criterion 
	it is clear that the trivial steady state of system~\eqref{model_resc_suma} is unstable.

	Note, that for $\omega\sigma\in(\pi/2,\pi)$ we have $\Re(W(i\omega))<0$ due to the 
	fact that on this interval cosine is negative and sine is positive. Consider 
	$\omega\sigma\in(0,\pi/2)$. Thus, cosine is positive and inequality 
	$\beta \omega\cos(\omega\sigma) - \gamma\sin(\omega\sigma)<0$ is equivalent to 
	\begin{equation}\label{tan}
		\frac{\beta}{\sigma\gamma} \omega\sigma < \tan(\omega\sigma).
	\end{equation}
	The right hand side of~\eqref{tan} is  a~strictly increasing convex function on $(0,\pi/2)$ and 
	have  the  first derivative at 0 equal to 1.  Thus,  the  first inequality of \eqref{cond:nstab:zabeka} 
	implies~\eqref{tan}.

	Now, it is enough to show that $g_R(\pi/\sigma) <0$. A~straightforth calculation yields
	\[
		-\left(\frac{\pi}{\sigma}\right)^2+\beta \frac{\pi}{\sigma} +\gamma <0. 
	\]
	Solving quadratic equation we get 
	\[
		\sigma < \frac{2\pi}{\beta+\sqrt{\beta^2+4\gamma}}\,,
	\]
	which is the second inequality of \eqref{cond:nstab:zabeka}.
\end{proof}

\begin{rem}
	Note, that \eqref{cond:stab:zabek}  implies $\sigma<\beta/\gamma$ and consequently, conditions \eqref{cond:nstab:zabeka} and \eqref{cond:stab:zabek}  are mutually exclusive. 
\end{rem}

\begin{thm}\label{prop:zabek:bif}
If the trivial steady state of system~\eqref{model_resc_suma} with  the  piecewise linear distributions defined 
by~\eqref{eq:zabek} and $\sigma\ge\varepsilon>0$ is locally asymptotically stable 
for $\sigma=\varepsilon$ and the function 
\[
 F(\omega) = \omega^4 - \varepsilon^2\Bigl(\omega^2\beta^2+\gamma^2\Bigr) g_1^2(\varepsilon \omega),
\]
where $g_1$ is defined by~\eqref{def:betagamag1} has no multiple positive zeros,
then the steady state is locally asymptotically stable for some $\sigma \in [\varepsilon,\sigma_0)$
and it is unstable for $\sigma>\bar\sigma$, with $\sigma_0\le \bar\sigma$. At $\sigma=\sigma_0$ the Hopf bifurcation 
occurs.

\end{thm}
\begin{proof}
 For the considered case the characteristic function is
given by~\eqref{char:zabek:sigmy}. 
It can be easily seen that if $W(i\omega)=0$, then 
\[
	\|-i\omega\|^2 = \varepsilon^2 \bigl\|i\omega \beta +\gamma\bigr\|^2
		 \Bigl\|g\bigl(i\omega\varepsilon\bigr)\Bigr\|^2,
\]
and thus
\[
\omega^4 = \varepsilon^2\Bigl(\omega^2\beta^2+\gamma^2\Bigr) g_1^2(\varepsilon \omega)\,,
\]
where $g_1$ is defined by~\eqref{def:betagamag1}.
The function $F$ has the following properties 
\[
	F(0) = -\varepsilon^2\gamma^2<0\,, \quad \text{ and } \quad 
	\lim_{\omega\to+\infty}F(\omega) = +\infty.
\]
Hence, we conclude that there exists $\omega_0>0$ such that $F(\omega_0)=0$ and $F'(\omega_0)\ge 0$. It is assumed 
that all zeros of $F(\omega)$ are simple, thus  $F'(\omega_0)>0$ and  by Theorem~1 from~\cite{cook86ekvacioj}, the 
thesis of theorem is proved.\end{proof}

\begin{rem}
	Note, that the case where the function $F$ has multiple positive zeros is not generic.
\end{rem}
\begin{proof}
	Consider the case of the multiple positive zero(s) of $F$.
First, note the following facts: 
\begin{enumerate}[(a)]
	\item $F$ is a~analytic function so it has isolated zeros;
	\item For $\omega>\max\{\varepsilon\sqrt{\beta^2+\gamma^2},1\}$ inequality $F(\omega)>0$ holds.
	\item If $g_1(\varepsilon\omega)=0$, then $F(\omega)=(2k\pi/\varepsilon)^4>0$ for some $k\in\N$,
				$k\ge 1$. 
\end{enumerate}
Facts (a) and (b) imply, that $F$ has a~finite number of zeros. Suppose that for the arbitrary values of $\beta_0$ and 
$\gamma_0$ the function $F$ has the multiple zero at $\omega_{m,j}$, for some $j=1,2,...j_M$, where $j_M\ge 1$ is a~natural number.
Now, consider the interval $I=[0,\max\{\varepsilon\sqrt{\beta^2+\gamma^2},1\}+1]$. Clearly, 
$\partial F/\partial \gamma <0$ holds for all $\omega\neq 2k\pi/\varepsilon$. Moreover, $F'$ is also an analytic function
so it has a~finite number of zeros in the interval $I$. Thus, we choose $0<\epsilon_{\gamma}<1$ so small that for all
$\gamma\in (\gamma_0-\epsilon_{\gamma},\gamma+\epsilon_{\gamma})\setminus\{\gamma_0\}$ the function 
$F$ has no multiple root in the interval $I$. This completes the proof.
\end{proof}

\section{Stability results for parameters estimated by Hahnfeldt~\etal}\label{sec:num}

In the previous section, we analytically 
investigated the stability of the trivial steady state of the family of angiogenesis models with distributed delays~\eqref{model_resc_suma} and distributions characterized by the probability densities 
$f_i$ given by~\eqref{eq:zabek} and~\eqref{eq:Erlang}--\eqref{eq:defE}. 
System~\eqref{model_resc_suma} is a~rescaled version of~\eqref{modeldis}, where the trivial steady state of~\eqref{model_resc_suma} corresponds to the positive steady state of~\eqref{modeldis}.

In this section we illustrate our outcome  for particular model parameters considered earlier in the literature, that is the parameters estimated by Hahnfeldt~\etal in the case of the model without the treatment (see ~\cite{hahnfeldt99cancer}). Namely, we consider the following set of parameters:
\begin{equation}\label{params}
\mu=0,\quad a_H=8.73\times 10^{-3},\quad b=5.85, \quad r=0.192,
\end{equation}
and we take the function $h(\theta) = -\ln \theta$. 
We present stability results for the distributed models for two different values of parameter $\alpha$. To keep the description short, for 
model~\eqref{model_resc_suma} with $\alpha=1$ we would refer to as  the Hahnfeldt~\etal model, while to
model~\eqref{model_resc_suma} with $\alpha=0$  as  the d'Onofrio-Gandolfi model. However, in this paper we also consider positive values of $\mu$ that can be interpreted as a~constant anti-angiogenic treatment.

\subsection{Erlang probability densities}
First, focus on model~\eqref{model_resc_suma} with  the  non-shifted Erlang distributions given by~\eqref{eq:Erlang}--\eqref{eq:defE} with $\sigma=0$, $i=1,2$. 
The expression $m/a$ describes the average delay. On the other hand, in the more
general case of $m_1\neq m_2$ a~ similar interpretation leads to the conclusion that the average delays are $m_1/a$ and 
$m_2/a$. One could also consider the arbitrary values $a=a_1$ for $f_1$ and $a=a_2$ for $f_2$, 
but then  the analytical expressions become very complicated and we decided not to study 
these cases here.

\begin{figure}[!htb]
	\centerline{\includegraphics{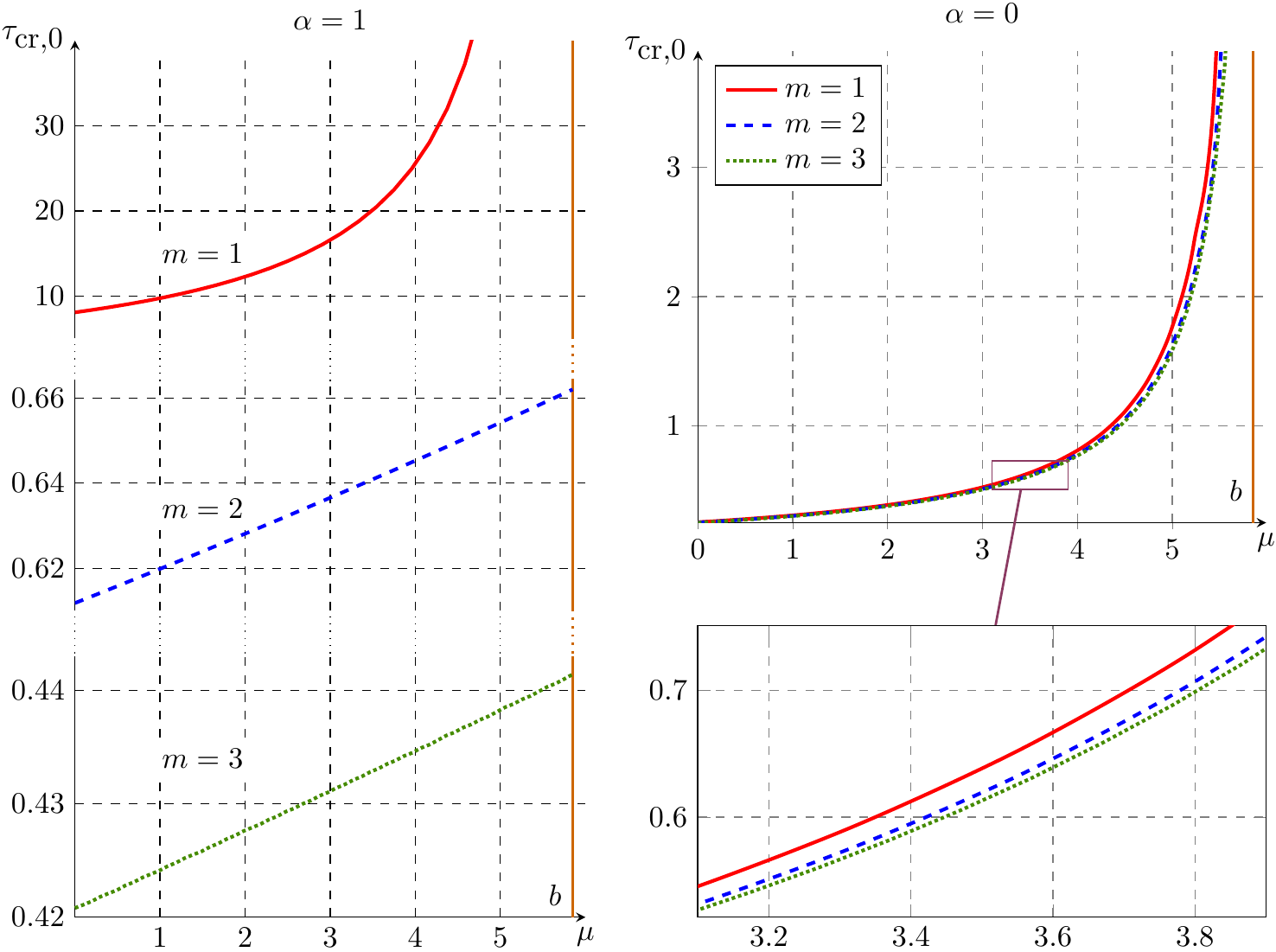}}
	\caption{Dependence of the critical average delay value in the case of the non-shifted Erlang distributions given by~\eqref{eq:Erlang}--\eqref{eq:defE}
		with $m_1=m_2=m$ on the constant treatment coefficient $\mu$ for system~\eqref{model_resc_suma}. 
		The  critical average delay is defined by $\tau_{\kryt,0} = m/a_{\kryt,0}$, where 
		$a_{\kryt,0}$ is the critical value of parameter $a$ for which stability change occurs (see 
		Proposition~\ref{thm:non-shiftE}). In the left-hand side panel results for 
		the Hahnletdt~\etal model ($\alpha=1$) are shown while in the right-hand side panel results 
		for the d'Onofrion-Gandolfi model ($\alpha=0$). All other parameters, except $\mu$,  are  defined by~\eqref{params}.
	\label{fig.zalodmu}} 
\end{figure}

\begin{figure}
\centerline{\includegraphics{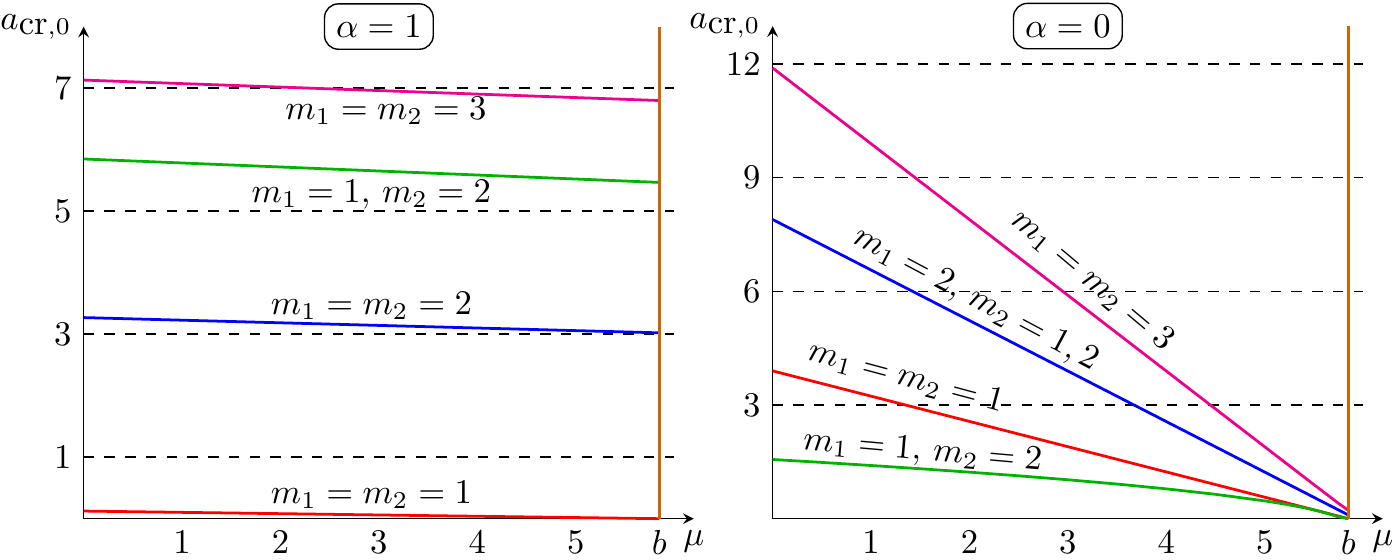}}
	\caption{Dependence of the critical value of parameter $a$ 
for the trivial steady state of system~\eqref{model_resc_suma} with the non-shifted Erlang probability densities given by~\eqref{eq:Erlang}--\eqref{eq:defE}  (i.e. $\sigma=0$). The regions above the plotted curves corresponds to the stability regions for the particular choices of parameters $m_i$, $i=1,2$. For $m_1=2$, $m_2=1$ and $\alpha=1$ the trivial steady state is always stable.  All other parameters, except $\mu$, are defined by~\eqref{params}.
	\label{fig.akr}} 
\end{figure}

In Fig.~\ref{fig.zalodmu}
the dependence of the critical average delay value $\tau_{\kryt,0}$
with $m_i=m$ ($i=1,2$) on the constant treatment coefficient $\mu$ is presented. 
In the case of the non-shifted Erlang distributions the average delay is defined by $\tau_{\kryt,0} = m/a_{\kryt,0}$ 
and the  left-hand side panel and the right-hand side panel present results for the Hahnfeldt~\etal 
model ($\alpha=1$) and the d'Onofrion-Gandolfi model ($\alpha=0$), respectively. The curves were
calculated from the stability conditions given in Proposition~\ref{thm:non-shiftE}. The stability regions of the trivial steady 
state of~\eqref{model_resc_suma} with  the non-shifted Erlang distributions are below curves, while above them there are the instability regions. 
It is worth to emphasise that there is a~qualitative difference between the cases: $m=1$ and $m=2,3$  for the Hahnfeldt~\etal model, 
while there is no such a~difference for the d'Onofrio-Gandolfi model. Clearly, for the Hahnfeldt~\etal model and $m=1$  the critical average delay
$\tau_{\kryt,0}$ tends to $+\infty$ as $\mu\to b=5.85$ while for $m=2$ and $m=3$ we have
 $\tau_{\kryt,0}\bigl|_{\mu=b}=4/\beta\approx 0.662$ and  $\tau_{\kryt,0}\bigl|_{\mu=b}=8/(3\beta)\approx 
0.441$, respectively.
Nevertheless, for the  d'Onofrio-Gandolfi model the critical average value of delays are almost the 
same for the case $m=1,2$ and $m=3$ for $\mu\in [0,b)$ and the behaviour is qualitatively comparable with 
the result for the Hahnfeldt~\etal model for $m=1$.
Moreover, in  the considered cases the critical average delay values in the case of  the non-shifted 
Erlang distributions are the~increasing functions of $\mu$. This means that the increase of treatment 
increases the stability region.

\begin{figure}[!htb]
\centerline{\includegraphics[width=0.9\textwidth]{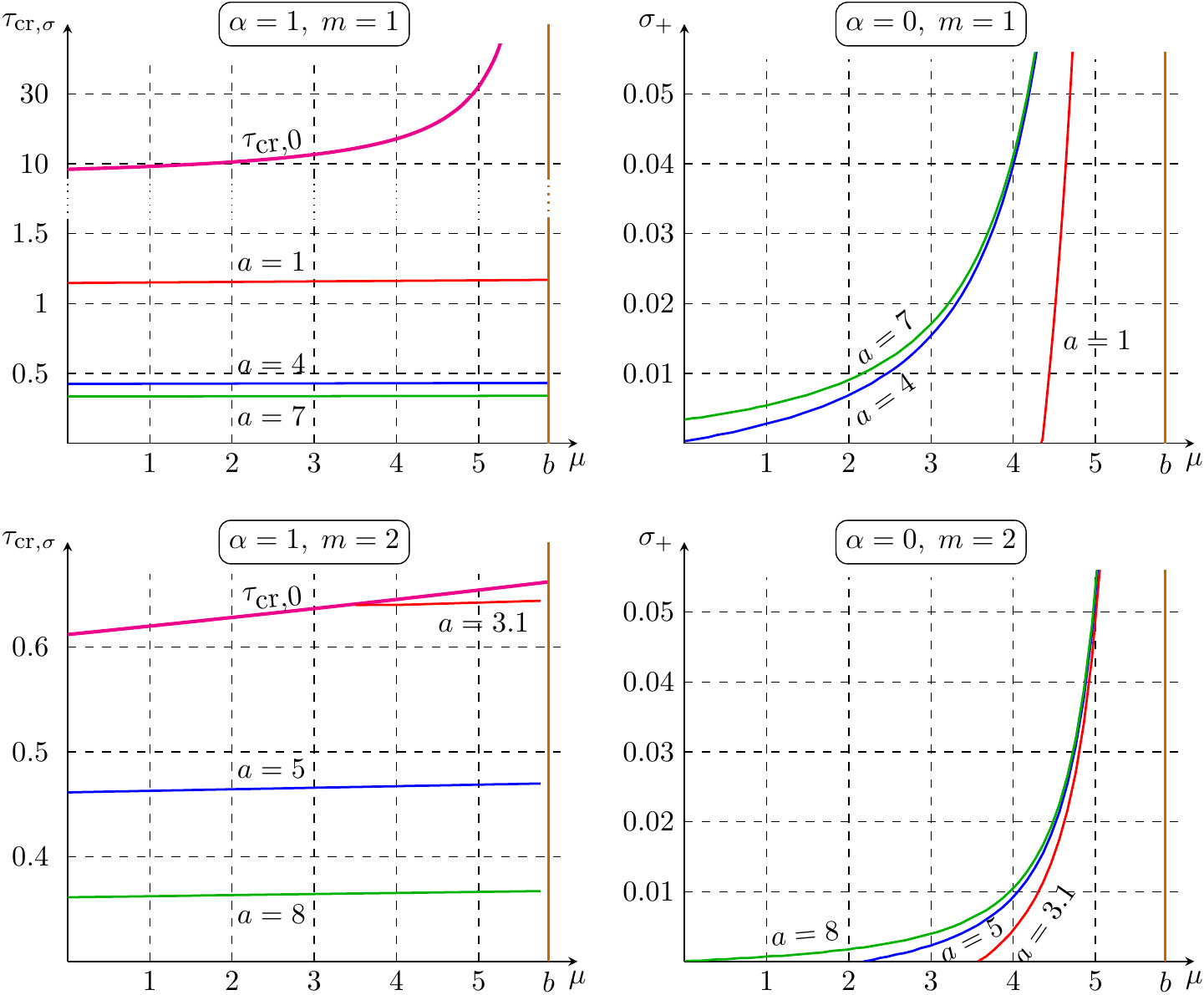}}
\bigskip{}
\centerline{\includegraphics[width=0.9\textwidth]{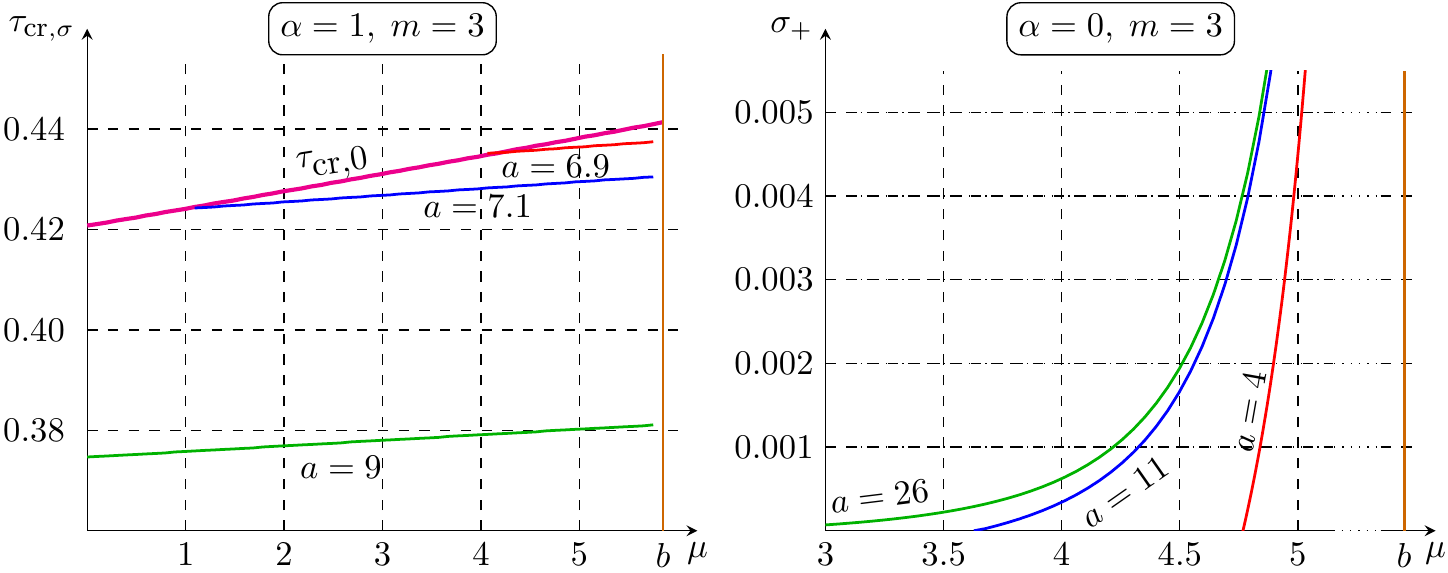}}
	\caption{Dependence of the critical average delay $\tau_{\kryt,\sigma}$ for the trivial steady state of 
		system~\eqref{model_resc_suma} with shifted Erlang probability densities given by~\eqref{eq:Erlang}--\eqref{eq:defE} for 
		the particular choice of parameter $a$ and comparison with the critical value of the average 
		delay $\tau_{\kryt,0}$ in the case of $\sigma=0$. For the case of the Hahnfeldt~\etal model 
		(i.e. $\alpha=1$), presented in the left column, the regions above the plotted curves correspond to 
		the instability regions for the particular choices of parameters $m_1=m_2$ and of parameter $a$. 
		The stability region is the region below curves.  For the d'Onofrio-Gandolfi model (i.e. $\alpha=0$) presented
		in the right column, the difference between critical value of the average delay $\tau_{\kryt,0}$
		in the case of $\sigma=0$ and the critical average delays $\tau_{\kryt,\sigma}$ for chosen values of $a$ are so small 
		that the plot would be non-informative. Thus, in the right column we presented 
		the difference $\sigma_+=\tau_{\kryt,0}$-$\tau_{\kryt, \sigma}$. 
		All other parameters, except $\mu$, are defined by~\eqref{params}.
	\label{fig.sigmakr}} 
\end{figure}

In Fig.~\ref{fig.akr} we plot the critical value of parameter $a$, that is $a_{\kryt,0}$, for the trivial steady state of system~\eqref{model_resc_suma} with the non-shifted Erlang probability densities given 
by~\eqref{eq:Erlang}--\eqref{eq:defE} (i.e. $\sigma=0$) calculated based on 
Proposition~\ref{thm:non-shiftE}. 
Here, 
the regions above the plotted curves correspond to the stability regions for the particular choices of parameters 
$m_i$, $i=1,2$, while those below correspond to the instability regions.
It should be clarified that for the case $\alpha=1$, $m_1=2$ and $m_2=1$ the steady state is stable 
for all values of $a$ .
We see that all plotted functions are 
decreasing as functions of the strength of constant treatment $\mu$. However, already for the non-shifted Erlang 
probability densities we observe a~large differences between both models regarding the model dynamics. 
For the Hahnfeldt~\etal
model we have the largest stability region for $m_1=2$ and $m_2=1$, while in the case of the 
d'Onofrion-Gandolfi model for $m_1=1$ and $m_2=2$. Moreover, for the case  $m_1=2$ and $m_2=1$ 
the steady state is 
stable independently of the parameter $a$ for  the  Hahnfeldt~\etal model, while in the same case of $m_i$ for the  
d'Onofrio-Gandolfi model the steady state might change its stability with the
change of $a$. Additionally, for $\alpha=0$ the sizes of the 
stability regions for  $m_i=2$ and $m_1=2$, $m_2=1$ are the same and it is not the case $\alpha=1$.

The stability of the steady state for the non-shifted Erlang distributions
can be determined by the Routh-Hurwitz criterion, although it may require tedious calculations, in particular for 
large values of parameters $m_1$ and/or $m_2$. On the other hand, the analytical results we obtained for the
 shifted Erlang 
distributions are rather limited. The results presented in Theorem~\ref{thm:Erlang} are only the existence results and give no information of the magnitude of the critical value of $\sigma$ for which stability is lost. Although, it is possible to derive an expression for the critical value of $\sigma$, it would involve $\arccos$ of 
some algebraic function of $\omega_0$, which value cannot be, in general, determined analytically 
and the final 
result would not be informative. In consequence, we decided to calculated  numerically the critical 
values of the average delay $\tau_{cr,\sigma}$   for the considered model with  the shifted Erlang 
distributions  for parameters given by~\eqref{params}.  
Clearly, if one wish to calculate the critical average delay for model with the shifted Erlang distributions it appears 
that it depends on sigma directly and is given by $\tau_{\kryt,\sigma}=m/a+\sigma_{\kryt}$,  where 
$\sigma_{\kryt}$ is a~critical value of $\sigma$ defined in Theorem~\ref{thm:Erlang}. 
In Fig.~\ref{fig.sigmakr} we present the stability results 
for the trivial steady state of system~\eqref{model_resc_suma} with the shifted Erlang probability densities given 
by~\eqref{eq:Erlang}--\eqref{eq:defE}. In particular, in Fig.~\ref{fig.sigmakr} (left column) we plot the
dependence of $\tau_{\kryt,\sigma}$ on the treatment strength $\mu$  
for the particular choices of parameters $m_i$ and the particular choices of parameter $a$ for the model with 
$\alpha=1$. Similarly as in Fig.~\ref{fig.zalodmu} the regions below the plotted curves correspond to the 
stability regions. Additionally, we plot $\tau_{cr,0}(\mu)$ curves to indicate 
thresholds for which the destabilizations occur in the case $\sigma=0$.  If  the curve for a~considered $a$ is above the $\tau_{cr,0}$ curve, then the steady state is unstable for $\sigma=0$ and remains unstable for all $\sigma>0$. In the both cases $m_i=1$ 
and $m_i=2$ the increase of the value of the parameter $a$ (which is equivalent to
decrease of the critical average delay), defining the shape of the probability densities, implies 
the decrease of the stability regions of the steady state. Since the corresponding curves for $\alpha=0$ are 
very close to each other instead of plotting them directly we decided to plot the differences  
$\sigma_+=\tau_{\kryt,0}-\tau_{\kryt, \sigma}$ between the critical values of the average delays $\tau_{\kryt,0}$ in 
the case  $\sigma=0$ and the critical average delays $\tau_{\kryt,\sigma}$ for the chosen values of $a$. Clearly, 
whenever plotted curve reaches the $\mu$ axis the trivial steady state becomes unstable for all $\sigma\ge 0$. 
From Fig.~\ref{fig.sigmakr} (right column) we deduce that the stability areas again decrease with the increase 
of the parameter $a$.

\subsection{Piecewise linear distributions}

Let us focus on the stability and instability regions of the trivial steady state of system~\eqref{model_resc_suma} with piecewise 
linear distributions defined by~\eqref{eq:zabek}. Since in such a~case we have smaller number of parameters defining the shape of the probability density we are able to plot the stability regions in ($\sigma$, $\varepsilon$) plane, see Fig.~\ref{fig.zabek}. Note that, 
for 
$\varepsilon>\sigma$  system~\eqref{model_resc_suma} becomes a~neutral system, which is out of our consideration hence this region is greyed.  Clearly, the estimations obtained in Proposition~\ref{prop:zabek:stab} are rough. However, solving numerically a~system $\Re W(i\omega)=0$, $\Im W(i\omega)=0$, where
real and imaginary part of characteristic function are given by~\eqref{reWzabek} and~\eqref{imWzabek}, respectively, 
we are able to calculate the stability region and the curve for which the stability change occurs. This line was plotted
in Fig.~\ref{fig.zabek} as a~thick solid blue line. 

 Additionally, we plot 
(in both panels) the vertical dashed lines that denote the conditions guaranteeing stability (see 
Proposition~\ref{prop:zabek:stab}(i)) or instability (see Proposition~\ref{prop:zabek:stab}(ii)) of the trivial steady 
state. For the  Hahnfeldt~\etal model  (left panel in Fig.~\ref{fig.zabek}) the 
condition~\eqref{cond:stab:zabek} from  Proposition~\ref{prop:zabek:stab} is denoted by a~vertical dashed line. 
In that case, the condition \eqref{cond:nstab:zabeka} cannot be fulfilled since 
the  expression on the right-hand side of \eqref{cond:nstab:zabeka} is always smaller than $\beta/\gamma$.
For the  d'Onofrio-Gandolfi model the dashed vertical line on the left to the solid 
vertical line indicates the condition~\eqref{cond:stab:zabek}. Two dashed vertical lines 
on the right to the solid one indicate the region in which the condition~\eqref{cond:nstab:zabeka} is fulfilled. Those numerical results show that
the conditions from  Proposition~\ref{prop:zabek:stab} are only sufficient but not 
necessary. From the figure we also deduce that the stability region for the trivial steady state for the   
d'Onofrio-Gandolfi model is smaller than the one for the distributed Hahnfeldt~\etal model. Moreover, 
computed by us the critical values of $\sigma$ for $\varepsilon\to 0$ and both $\alpha=1$ 
and $\alpha=0$ agrees with those calculated by Piotrowska~and~Fory{\'s} in \cite{Piotrowska2011} for the models 
with discrete equal delays. That agrees with the intuition, since for $\varepsilon\to 0$ models given by~\eqref{model_resc_suma} with piecewise 
linear distributions defined by~\eqref{eq:zabek} reduce to the models with  double discrete delay.
\begin{figure}[!htb]
	\centerline{\includegraphics{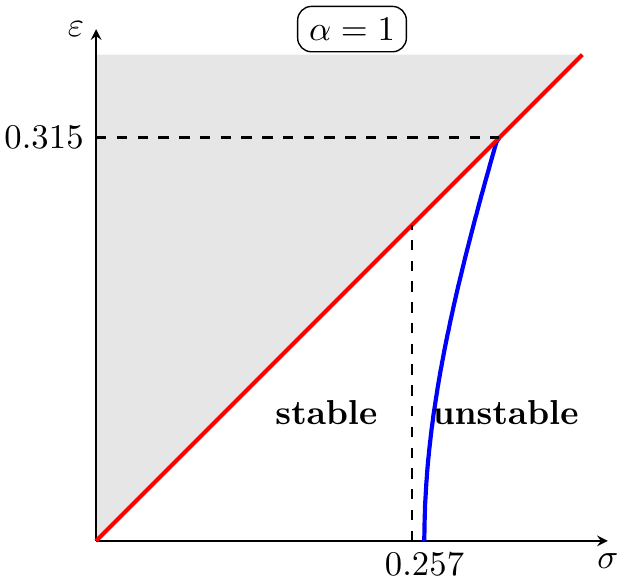}\;\includegraphics{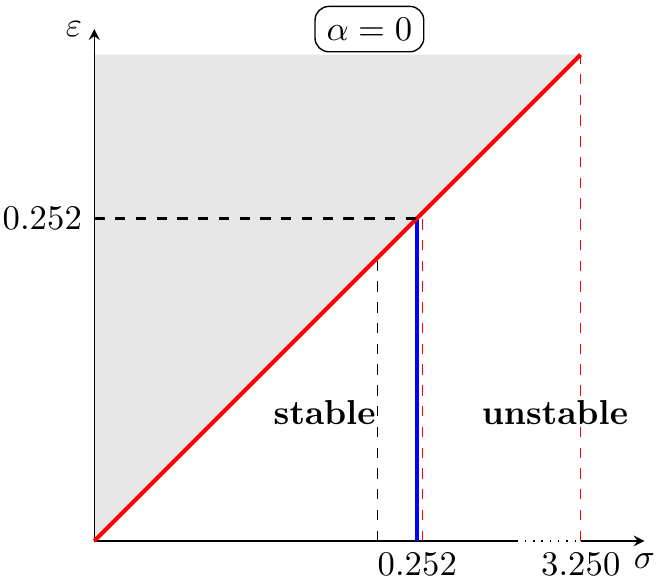}}
	\caption{Stability and instability of the trivial steady state of system~\eqref{model_resc_suma} with piecewise 
	linear distributions defined by~\eqref{eq:zabek} with $\sigma>0$, $\varepsilon>0$, $i=1,2$
	for  the parameters given by~\eqref{params} with $\alpha=1$ (the left hand-side panel) and $\alpha=0$ (the right-hand side 
	panel). The greyed part denotes the region where $\varepsilon>\sigma$
	for which system~\eqref{model_resc_suma} becomes a~neutral system. The solid lines indicate the stability switch. In the left-hand side panel, the vertical dashed line denotes the 
	condition~\eqref{cond:stab:zabek} from Proposition~\ref{prop:zabek:stab}.
	In the right-hand side panel, the vertical dash line in the stability region denotes  
	the scope of the condition~\eqref{cond:stab:zabek} form Proposition~\ref{prop:zabek:stab}, while two  vertical  dash lines in the instability region denote  
	the scope of condition~\eqref{cond:nstab:zabeka}. 
	\label{fig.zabek}} 
\end{figure}

\begin{figure}[!htb]
	\centerline{\includegraphics{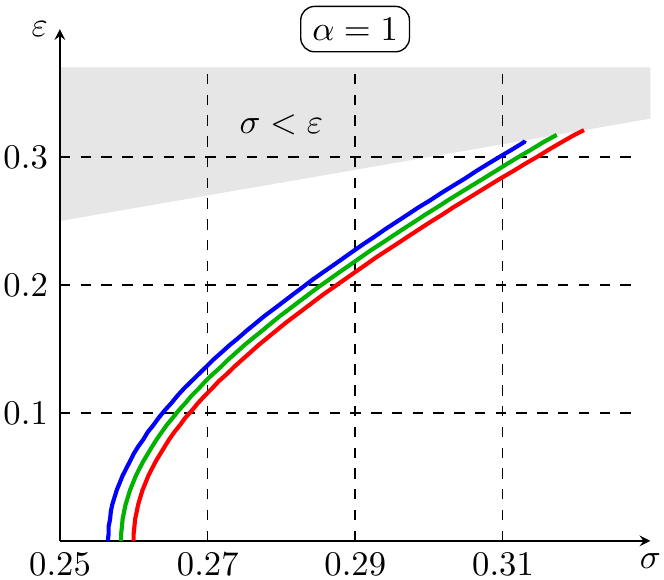}\;\includegraphics{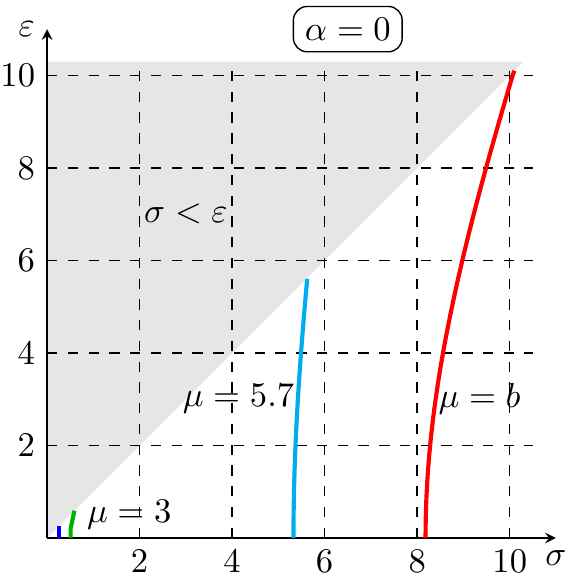}}
	\caption{Stability and instability of the trivial steady state of system~\eqref{model_resc_suma} with piecewise 
	linear distributions defined by~\eqref{eq:zabek} with $\sigma>0$, $\varepsilon>0$, $i=1,2$
	for the parameters given by~\eqref{params} with $\alpha=1$ (the left hand-side panel) and $\alpha=0$ (the right-hand side 
	panel) and different $\mu$. The greyed part denotes the region where $\varepsilon>\sigma$
	for which system~\eqref{model_resc_suma} becomes a~neutral system. The solid lines indicate the stability switch
	for (from left to right) $\mu=0,3,b$ for $\alpha=1$ and for $\mu=0,3,5.7,b$ for $\alpha=0$. \label{fig.zabekmu}} 
\end{figure}

For $\mu>0$ and the Hahnfeldt~\etal model we observe a~move of the stability switch curve to the right with hardly any change in the shape indicating that the stability region increases with increase of parameter $\mu$, see left panel in  Fig.~\ref{fig.zabekmu}. This shift is very small. Moreover, for $\mu=b$ we obtain limiting values: $0.26$ for $\varepsilon\to 0$ and $0.33$ for 
$\varepsilon\to \sigma$. On the other hand, for the d'Onofrio-Gandolfi model the change is more visible. The lines start to lean to the right. For small $\mu$, the change is quite small, compare with right panel in Fig.~\ref{fig.zabek}. For example, 
for $\mu=3$ we obtain $\sigma=0.509$ for $\varepsilon\to 0$ and $\sigma = 0.59$ for $\varepsilon=\sigma$. As $\mu$ approaches to $b$ the changes become more rapid. For $\mu=5.7$ we have $\sigma=5.326$ for $\varepsilon\to 0$ and $\sigma = 5.632$ for $\varepsilon=\sigma$, while for $\mu=b$ we have $\sigma=8.181$ and $\sigma =10.065$ for 
 $\varepsilon\to 0$  and $\varepsilon=\sigma$, respectively.



\section{Discussion}\label{sec:dis}

Although discrete delays often
appear in models describing various biological phenomena, 
e.g.~\cite{bocharov2000,Erneux09,mbab12nonrwa,monk03currbiol,piotrowskaMBEmut,Piotrowska2012RWA,uf02jtb,mbufjp11jmaa,jmjpmbuf11bmb,cooke99jmb,enatsu11dcdsB}
and references therein, it is 
obvious that in natural processes delay, if it exists, is usually somehow distributed around some average 
value 
due to differences between individuals and/or environmental noise. Thus, we believe that  distributed 
delays 
are more realistic. However, in some cases, when the distribution has small variance and  compact support discrete delays may be treated as a~good approximation. 
Clearly, models with discrete and distributed delays might have different dynamics in some ranges of parameters. 
In the presented paper we have compared the behaviour of two types of systems in the context of the 
stability of the steady state for a~particular family of models  describing the tumour angiogenesis process.  

\begin{table}[!htb]
\caption{The critical average delay value $\tau_{\kryt,0}$ in the case of non-shifted Erlang distributions given 
			by~\eqref{eq:Erlang}--\eqref{eq:defE} with $m_1=m_2=m$ and $a_1=a_2=a$ 
		for arbitrary chosen parameters $\mu$ for system~\eqref{model_resc_suma}. 
		The critical average delay is defined by $\tau_{\kryt,0} = m/a_{\kryt,0}$, where 
		$a_{\kryt,0}$ is the critical value of parameter $a$ for which  the stability change occurs (see 
		Theorem~\ref{thm:non-shiftE}).  }
\centerline{
\begin{tabular}{l|r|r|r|r||r|r|r|r|}
		 &\multicolumn{4}{|c||}{Hahnfeldt~\etal model ($\alpha=1$)} 
		 &\multicolumn{4}{c|}{d'Onofrion-Gandolfi model ($\alpha=0$) } \\ \hline
$\mu$		 & $0$ & $2$ & $4$ & $5.85$ & $0$ & $2$ & $4$ & $5.85$\\ \hline
$m_1=m_2=1$ & 8.069 & 12.261 & 25.515 & $\infty$
				& 0.256 & 0.390 & 0.811 & $\infty$	\\ \hline
$m_1=m_2=2$ & 0.611 & 0.628 & 0.645 & 0.662
				& 0.253 & 0.382 & 0.780 & 20.833			\\ \hline
$m_1=m_2=3$ & 0.421 & 0.428 & 0.435 & 0.441
				& 0.252 & 0.380 & 0.770 & 13.889			\\ \hline
\end{tabular}}\label{tab:kr}
\end{table}

Presented analysis of distributed models includes: the uniqueness, the positivity and the global existence of 
the solutions, the existence of the steady state and the possibility of the existence of the stability switches. 
We have analytically derived conditions, involving the parameters defining the probability densities, 
guaranteeing 
the stability or instability of the steady state. We have also shown that in some cases the single 
stability switch is observed and the Hopf bifurcation occurs. 

For the particular set of parameters, estimated by Hahnfeldt~\etal, we have investigated the stability regions for steady state. We compared the results for both the  
Hahnfeldt~\etal ($\alpha=1$) and the  d'Onofrio-Gandolfi ($\alpha=0$) models in the case of different 
probability densities. For both models we considered the Erlang shifted and non-shifted distributions as 
well as the piecewise linear distributions.
We want to emphasise here that, in the general case, it is hard to say for which model Hahnfeldt~\etal or d'Onofrio-Gandolfi the 
stability region is larger since it strongly depends on the considered probability densities and theirs shapes.  
However, we observe certain similarities. 

First, for $\mu=0$, i.e. the family of models without treatment, we see that for the Hahnfeldt~\etal model with 
non-shifted Erlang probability density the larger $m_i=m$ are the smaller the stability region 
is, see Fig.~\ref{fig.zalodmu} 
 and Table~\ref{tab:kr}. The same holds for the d'Onofrio- Gandolfi model with the non-shifted 
Erlang distributions, see Table~\ref{tab:kr} and moreover, for all considered $m_i=m$, $i=1,2$ (and $\mu=0$)
for the Hahnfeldt~\etal model the stability region is larger than for d'Onofrion-Gandolfi one. 
Similarly, for the models with the piecewise linear distributions the stability region for $\mu=0$ 
for the Hahnfeldt~\etal model is larger than the one for the d'Onofrio-Gandolfi model, compare Fig.~\ref{fig.zabek}.

If we consider a~positive parameter $\mu$ smaller than $b$ 
(to ensure the non-negativity of the  steady state of model~\eqref{modeldis}, which corresponds to the trivial 
steady state of 
\eqref{model_resc_suma}), we observe a~similar tendency for the Erlang distributions in dependence on the 
parameters $m_1=m_2$
(compare Table~\ref{tab:kr} for the non-shifted distribution case).
However, the dependence on $\mu$ depends strongly on the chosen model. For the Hahnfeldt~\etal model 
and $m_1=m_2\ge 2$, this 
dependence is almost linear and very weak, while for the d'Onofrio-Gandolfi model it is much stronger. In 
results, for  $m_1=m_2\ge 2$ and any sufficiently large value of $\mu\in[0,b)$ the 
critical average delay for  the Hahnfeldt~\etal model becomes smaller than for the d'Onofrio-Gandolfi 
model. 
The case $m_1=m_2=1$ is different. Dependence on $\mu$ is similar for both 
models and 
the critical average delay for  the Hahnfeldt~\etal model stays larger than for the 
d'Onofrio-Gandolfi model for any given value of $\mu$. Similarly, for different values of $m_i$, it 
seems that dependence on $\mu$ is stronger for the d'Onofrio-Gandolfi model,  see Fig.~\ref{fig.akr}.
Nevertheless, for the d'Onofrio-Gandolfi model with the non-shifted Erlang distributions there is no difference 
regarding the stability results between cases
$m_1=m_2=2$ and $m_1=2$, $m_2=1$, which is not the case for the Hahnfeldt~\etal model, where for 
$m_1=2$, $m_2=1$ we have stability independently on the value of parameter $a$.
For the shifted Erlang distributions we have investigated the changes of the size of the stability region also in the 
context of the change of  the value of parameter $a$. For both the Hahnfeldt~\etal and  d'Onofrion-Gandolfi 
models we see that for all considered $m=m_i$, $i=1,2$, the increase of the parameter $a$ decrease the stability 
region, however in the case of d'Onofrio-Gandolfi model we 
have compared the differences of the $\tau_{cr,0}-\tau_{cr, \sigma}$.
A~similar increase of the stability region with 
a decrease of concentration of delay distribution was observed in~\cite{Bernard2001} for one linear equation with distributed delay and Erlang probability density. 

Nevertheless, we see that since for all considered cases of the shifted and non-shifted 
Erlang distributed models $\tau_{cr,\sigma}$ and $\tau_{cr,0}$, respectively, are increasing (sometimes 
slightly) functions of variable $\mu$ implying that the increase of the constant
treatment strength enlarges the stability regions for the positive steady state of model~\eqref{modeldis}. Our analysis also shows that the increase of the positive parameter $\mu$ 
enlarges the stability area for the steady state for both (Hahnfeldt~\etal and d'Onofrion-Gandolfi) distributed models with the piecewise linear distributions, but this time for the d'Onofrion-Gandolfi model this increase is more pronounced than for the distributed Hahnfeldt~\etal model, see Fig.~\ref{fig.zabekmu}. 
Moreover, for small values of parameter $\mu$ the stability region for the Hahnfeldt~\etal 
model with the piecewise linear distributions is larger than the one for the d'Onofrio-Gandolfi 
model, while for the larger values of $\mu$ the situation is opposite. Performed simulations show that 
the change occurs before $\mu \approx 1.42$. 

The variances for the shifted and non-shifted Erlang distributions are given by $\frac{m_i}{a^2}$ and 
they give a~measure of the degree of the concentration of the delay around the mean. 
Actually, a~better measure of the spread of the distribution around the mean for our purposes is the 
coefficient of variation, i.e. the ratio of the standard deviation to the mean, that is $\sqrt{m_i}/(a\sigma+m_i)$. 
Clearly, for the non-shifted distributions 
we have  $1/\sqrt{m_i}$, which implies that larger parameter $m_i$ yields smaller the 
considered ratio. Hence, the increase of $m_i$ decreases the percentage dispersion of the average delay. 
On the other hand, for the shifted Erlang 
distributions, the coefficient of variation is a~decreasing function of $\sigma$. This dependence is obvious since the increase of $\sigma$ with fixed $m_i$ and $a$ means that the average delay is increased while standard deviation remains constant. The coefficient of variation is also a~decreasing function of parameter $a$. On the other hand, if we fix $a$ and $\sigma$, and study the influence of $m_i$, then the coefficient of variation increases if $m_i<a\sigma$ and decreases otherwise.

For piecewise linear distributions ($\sigma\geq\varepsilon$) we have the average value of $f_i$ (defined by 
\eqref{eq:zabek}) equal to $\sigma$ and the standard deviation given by $\varepsilon/\sqrt{6}$. Hence,  the coefficient of variation equals to $\varepsilon/(\sigma\sqrt{6})$. Thus, it is an~increasing 
function of $\varepsilon$ (for fixed $\sigma$) and  a~decreasing function of $\sigma$ (for fixed $\varepsilon$). 
Thus, we conclude that the percentage dispersion of the average delay is the increasing function of 
$\varepsilon$ and decreasing function of $\sigma$ and it is always smaller than 1. 
Clearly, all that should be taken into account whenever the probability distributions describing the 
characteristics of delays are estimated.

We believe that with the proposed type of models one can describe in a~more realistic way  the process of angiogenesis. However, the considered model should be validated with the experimental data. 
Our results show that a~system behaviour in the case of distributed delays 
might be different from the behaviour of the system for discrete delays. Hence, to validate the model with the experimental data
it is important to choose the right type of the delay(s) in the model preferably based on the experimental data or 
hypothesis postulated by the experimentalists. Our study also shows that a~shape of probability density for the 
distributed delays has an essential influence on the model dynamics. On the other hand, it is not a~trivial task 
to choose the right distribution if one has not sufficiently large set of the experimental data and does not 
assume any particular shape of the distribution. This problem need to be individually addressed according the 
available data and it should be done for the considered model in the future.
In our opinion, models with distributed delays could be also studied in the context 
of efficiency of different types of tumour  therapies. 

 We would like to emphasise that oscillations in the experimental data for 
different human and animal cell lines in the context of neoplastic diseases, even without administration of any 
treatment, has been already widely observed for leukaemia 
(\cite{Mackey1978,Fortin1999}), B-cell lymphoma (\cite{Uhr1991}) and solid tumours (\cite{Chignola2000, Chignola2003}). 
To investigate the origin of such phenomenon a~number of mathematical and numerical models has been developed. For 
example Kuznetsov et al., \cite{Kuznetsov1994}, studying the interactions between an immune system 
and a~tumour by system of ODEs, pointed out that different local and global bifurcations 
(observed for the realistic parameter values they estimated) 
 showed that there might be a~connection between the phenomena of immunostimulation 
of tumour growth, formation of ``dormant'' state and ``sneaking through'' of tumour. Similarly, Kirschner and Panetta in~\cite{Kirschner1998} discussed
inter alia the biological implications of the Hopf bifurcation in the case of the tumour-immune system interactions. More recently, Pujo-Menjouet et al.~\cite{Pujo-Menjouet2005}, 
using the model with discrete delay for which the Hopf bifurcation is observed,
explained why and how short cell cycle durations gave rise to long period oscillations 
of order form 40 to 80 day for periodic chronic myelogenous leukaemia patients. 
At the same time in \cite{Bernard2004} the bifurcation phenomenon and its role for a~white-blood-cell
production model with discrete delay was studied. The experimentally observed oscillations in the tumour volume~\cite{Chignola2000, Chignola2003}
can be explained by the analytic results from earlier and more theoretical paper~\cite{Byrne1997} focussing on the 
growth of avascular multicellular tumour spheroids, where the delay in the process of proliferation was
considered.  Later on, this work was extended in \cite{mbuf03mcm_prol}, while the model with
the delay only in the  apoptosis  process was considered
in~\cite{mbuf03mcm_reg}. Next, the model that took into account delays in both processes was studied in~\cite{piotrowska2008a}. 
It is believed that the 
p53 protein plays an essential role in preventing certain types of cancers and is called ``guardian of the genome''~\cite{Lane1992nature}, hence we could also mention here the models of Hes1 and p53 gene 
expressions, both with negative feedback, and with discrete time delays proposed by Monk \cite{monk03currbiol}. The model of the Hes1 protein was examined later by 
Bernard et al.~\cite{Bernard2006} and Bodnar and Bart\l{}omiejczyk~\cite{mbab12nonrwa}.In that cases the mathematical 
analysis showed, that experimentally observed oscillation can be explained by the time lag present 
in the system due to the DNA~transcription process as well as the diffusion time of mRNA~and the Hes1 proteins.

Finally, we discuss shortly possible generalisation of the results obtained in this paper.
The first equation of the considered model is quite general and can describe a~process with saturation. 
On the other hand, the second equation comes from a~quasi-stationary approximation of some reaction-diffusion 
equations under particular assumptions that are justified for the angiogenesis process. Thus, most probably 
it cannot be straightforwardly transformed to describe other biological or medical processes. However, the
 stability analysis performed in the paper is based on the local properties of the functions. Thus, we may 
assume a~more general form of the second equation namely $q(t) G(p_t/q_t, p)$ (where 
$p_t$ and $q_t$ according to a~functional notation typical for DDEs denote terms with delay) assuming that $G$ is an 
increasing function in the first variable and decreasing in the second one. Then, in our opinion, under some 
additional assumptions, one should obtain results similar to those presented in the paper.


\appendix
\section{Generalized Mikhailov Criterion}\label{appen}

Here we formulate a~generalized version of the Mikhailov Criterion (see eg.~\cite{uf04jbs}). The
classical formulation of the Mikhailov Criterion is for the characteristic functions that are 
sums of polynomials multiplied by an exponential function. However,  it can be 
generalized for wider class of functions. Below, we present a~detailed formulation.

\begin{thm}[Generalised Mikhailov Criterion]
	Let us assume that $W:\C\to\C$ is an analytic function, has no zeros on imaginary axis 
	and fulfils 
	\begin{equation}\label{charfun}
		W(\lambda) = \lambda^{n} +o(\lambda^{n}), \quad 
		W'(\lambda) = n\lambda^{n-1} + o(\lambda^{n-1}),
	\end{equation}
	for some natural number $n$. Then the number of zeros of $W$ in the 
	right-hand complex half-plane is equal to $n/2-\Delta/\pi$
	where 
	\[
		\Delta = \Delta_{\omega\in[0,+\infty)} \text{arg} W(i\omega). 
	\]
\end{thm}
Note, that $\Delta$ denotes the change of argument of the vector $W(i\omega)$ in the positive direction
of complex plane as $\omega$ increases from $0$ to $+\infty$. 
\begin{proof}
	The proof is exactly the same as the proof of the Mikhailov Criterion in~\cite[Th.~1]{uf04jbs}. 
	It is based on integration of the characteristic function on the following contour: part of imaginary axis,
	portion of the imaginary axis from $i\rho$ to $-i\rho$ and a~semi-circle of radius $\rho$  with the middle in zero 
	located in the right half of complex plane.
	The fact that $W$ is an analytic function implies that it has the isolated zeros and 
	all integration can be done in the same way as in~\cite{uf04jbs}. Moreover, the
	assumption~\eqref{charfun} implies that all limits calculated in~\cite{uf04jbs}
	remain the same. 
\end{proof}

\section{Acknowledgements}
The work on this paper was supported
by the Polish Ministry of Science and Higher Education, within the Iuventus Plus Grant: "Mathematical
modelling of neoplastic processes" grant No. IP2011 041971.  

We thank the reviewers for their valuable comments that improved our paper.



\end{document}